\documentclass[journal,twoside]{IEEEtran}

\usepackage[table]{xcolor}
\usepackage{xcolor}
\usepackage{amsmath,amsfonts}
\usepackage{algorithmic}
\usepackage{algorithm}
\usepackage{array}
\usepackage[caption=false,font=normalsize,labelfont=sf,textfont=sf]{subfig}
\usepackage{textcomp}
\usepackage{stfloats}
\usepackage{url}
\usepackage{verbatim}
\usepackage{graphicx}
\usepackage{cite}
\usepackage[thinc]{esdiff}
\usepackage{multirow}
\usepackage{booktabs}
\usepackage[margin=1in]{geometry}

\usepackage[normalem]{ulem}

\IEEEoverridecommandlockouts

\usepackage{nicefrac}
\usepackage{subfig}

\usepackage{pgfplots}
\usepackage{tikzscale}
\usepackage{siunitx}
\usepackage{multirow}
\usepackage[font=small]{caption}
\usepackage[font=small]{subcaption}

\usepackage{hyperref}
\usepackage[capitalise]{cleveref}

\usepackage{enumitem}

\usepackage{amssymb}
\usepackage{pifont}
\usepackage{booktabs}

\usepackage{amsthm}

\usepackage{optidef}
\usepackage{enumitem}

\newtheorem{rem}{Remark}
\newtheorem{assum}{Assumption}
\newtheorem{definition}{Definition}
\newtheorem{prop}{Proposition}
\newtheorem{prob}{Problem}

\newtheorem{theorem}{Theorem}

\newtheorem{lemma}{Lemma}

\newcommand{\distime}{t}
\newcommand{\distimenext}{t+1}
\newcommand{\step}{k}
\newcommand{\stepnext}{k+1}
\newcommand{\firstagent}{i}
\newcommand{\secondagent}{j}

\newcommand{\filledcircle}[2][gray!80,fill=gray!80]{\tikz[baseline=-0.5ex]\draw[#1,radius=#2] (0,0) circle ;}%

\newcommand{\filledsquare}[2][fill=gray!80,draw=gray!80]{%
  \tikz[baseline=-0.55ex] \node[
    inner sep=0pt,
    minimum size=#2,
    #1
  ] (sq) {};
}

\definecolor{SetPurple}{HTML}{9673A6}
\definecolor{SetBlue}{HTML}{6C8EBF}
\definecolor{ContractPurple}{HTML}{9933FF}
\definecolor{ContractBlue}{HTML}{3333FF}

\makeatletter
\def\@opargbegintheorem#1#2#3{\trivlist
   \item[]{\bfseries #1\ #2\ (#3)} \itshape}
\makeatother

\usepackage{mathtools}

\usepackage{color}

\title{Anytime Plug-and-Play Control with \\ Contract-Based Distributed MPC 
}

\author{Sabrina Bodmer, Danilo Saccani, Melanie N. Zeilinger, Andrea Carron
\thanks{This research has been supported by the Swiss National Science Foundation under the NCCR Automation (grant 51NF40\_225155).}
\thanks{S. Bodmer, A. Carron and M.N. Zeilinger are with the Institute for Dynamic Systems and Control, ETH Zurich, Switzerland. (email \texttt{\{sabodmer,carrona,mzeilinger\}@ethz.ch)}  }%
\thanks{D. Saccani is with the Institute of Mechanical Engineering, École Polytechnique Fédérale de Lausanne (EPFL), CH-1015 Lausanne, Switzerland. (email: \texttt{danilo.saccani@epfl.ch})}%
}
\usepackage{etoolbox}

\makeatletter
    \apptocmd{\@maketitle}{
        \centering
        \captionsetup{type=figure}
        \vspace{5mm}
         \includegraphics[width=1.0\linewidth]{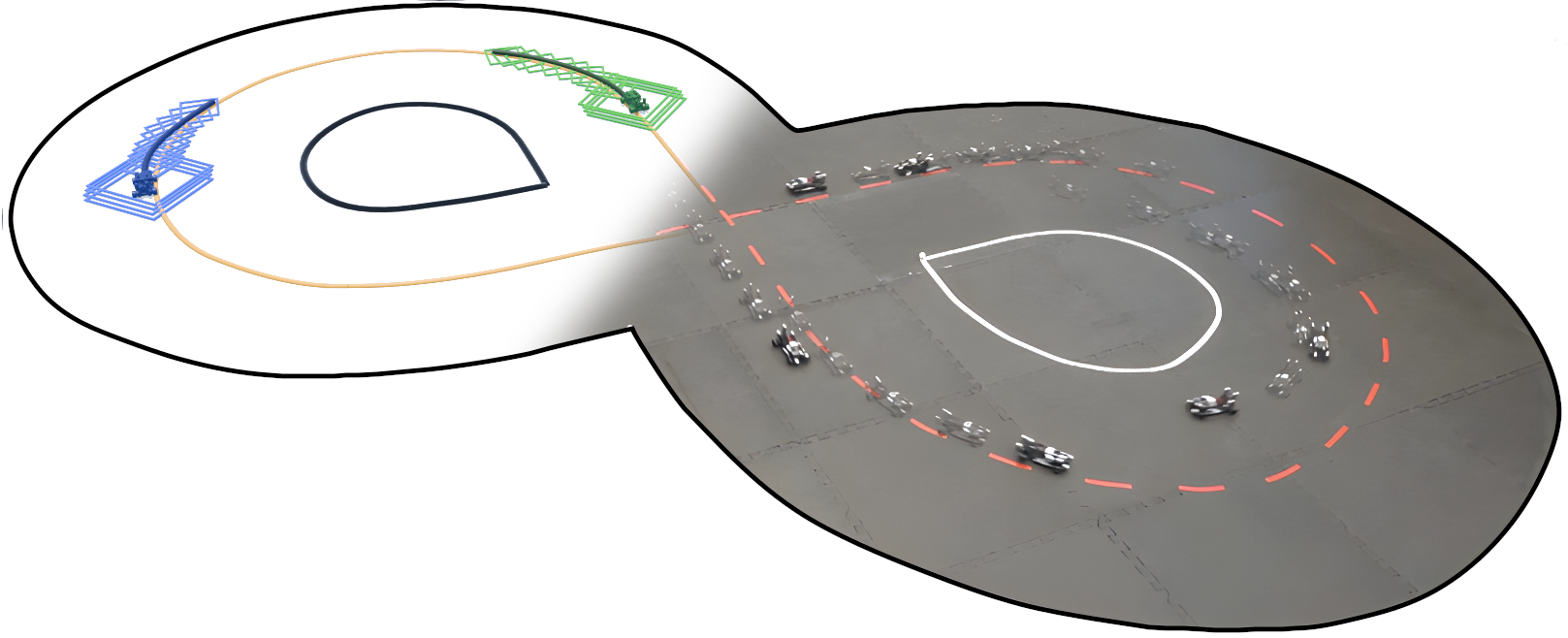}
        \captionof{figure}{Simulation and hardware experiments conducted using up to eight small-scale race cars on a figure-eight track.}
        \vspace{5mm}
        \label{fig:hardware_overview}
        \setcounter{figure}{0} 
        \vspace{-20.5pt}
    }{}{}%
\makeatother
\begin{document}
\maketitle
\setcounter{figure}{1} 

\begin{abstract}

A central challenge in many mobile multi-robot applications is that communication topologies are inherently time-varying. Agents may enter or exit the network and such changes cannot generally be restricted a priori.
This work introduces a distributed multi-agent control algorithm based on local communication that supports anytime agent joining and leaving the communication network without centralized coordination.
The method scales efficiently with the number of agents by relying on a distance-based neighbor definition and on contracts derived from predicted trajectories. The resulting contract constraints guarantee collision avoidance and constraint satisfaction.
We validate the proposed method in an autonomous multi-agent driving scenario, demonstrating effective collision avoidance in high-speed, dynamic environments with agents moving in opposite directions, in both simulated and real-world experiments.
\end{abstract}


\IEEEpeerreviewmaketitle

\noindent {\small {\bf Supplementary Material/Videos}: \url{anytime-plug-and-play.github.io}}

\section{Introduction}
\label{sec:Introduction}
\IEEEPARstart{A}{s} multi-agent robotic systems, including drones and autonomous vehicles, are increasingly being deployed in real-world applications,  efficient and safe coordination among agents has become a critical requirement~\cite{bagloee2016autonomous}. Model Predictive Control (MPC) is a popular framework for multi-agent coordination~\cite{centralized_mpc_riegger2016, centralized_mpc_hajar_2016}. Still, centralized MPC suffers from poor scalability due to cubic computational growth~\cite{centralized_comp_frejo_2012} and its reliance on global information and global communication. To overcome these limitations, distributed MPC (DMPC) approaches allow each agent to solve a local problem while exchanging limited information, such as predicted trajectories~\cite{FARINA20121088}, planned inputs~\cite{STEWART2010460}, or abstract contracts~\cite{Lucia2015contract}, with neighboring agents over a communication graph~\cite{SCATTOLINIreview}. This improves scalability and preserves privacy, but the ability of agents to coordinate effectively depends strongly on how the communication network changes over time. The way networks evolve can be categorized into three common types: 1) fixed, with constant agents and communication, 2) time-varying, with constant agents but changing neighbors, 3) plug-and-play (PnP), where agents, neighbors, and communication can all change dynamically~\cite{saccani2023model, Lucia2015contract, Zeilinger2013plugplay, riverso2012}. PnP networks can be further divided into request-based, where agents may enter or leave upon approval of the entire or part of the network\cite{Zeilinger2013plugplay, riverso2012}, and anytime, where agents may enter or leave the network at any time without the other agents' approval. The former suits systems with tightly coupled dynamics, such as power grids~\cite{riverso2012,Zeilinger2013plugplay}. At the same time, the latter reflects settings such as mobile robots, where dynamics are decoupled but constraints and communication depend on position. In such environments, controllers must ensure safe operation despite arbitrary changes in the network topology. 
Most distributed MPC methods address fixed or request-based plug-and-play topologies, but do not support uncoordinated topology changes while retaining formal safety guarantees. Our previous work~\cite{saccani2023model} considered multi-agent MPC under limited communication and a time-varying communication topology. However, the optimization problem in~\cite{saccani2023model} was formulated jointly over all agents in each connected component and therefore remained centralized and coupled at the formulation level. Although distributed numerical solution methods were discussed as a possible implementation route,~\cite{saccani2023model} did not derive or evaluate independent local MPC problems. In contrast, the present work introduces time-varying cells and safety envelopes that decouple the collision-avoidance constraints, allowing every agent to solve an independent local FHOCP after a single exchange of predicted trajectories with its current neighbors. This yields a non-iterative distributed implementation with recursive-feasibility and collision-avoidance guarantees under anytime plug-and-play operation.

\subsubsection*{Contributions}
This paper makes three main contributions.
First, we propose a non-iterative distributed MPC framework that guarantees collision avoidance and constraint satisfaction under plug-and-play operations without requiring approval from other agents (anytime plug-and-play).
The key tools are time-varying cells and safety envelopes, enabling fully decoupled collision-avoidance constraints via a contract-based mechanism. 
Second, we show how to augment the safety-guaranteed framework for a given autonomous driving example, preserving guarantees while enabling aggressive maneuvers.
Third, we validate the approach in high-fidelity simulation and real-world experiments with six to eight autonomous race cars on a figure-eight track (Fig.~\ref{fig:hardware_overview}), demonstrating applicability in high-speed, dynamic settings with frequent topology changes.

\subsubsection*{Related works}
When the communication network is fixed, distributed optimization methods such as dual ascent~\cite{dual_descent_giselsson, dual_descent_necoara} and Alternating Direction Method of Multipliers (ADMM)~\cite{admm_boyd, admm_rey} are widely used to coordinate local subproblems. These approaches are well-suited for sparse couplings but often need many iterations and are biased toward linear or convex formulations, which limits real-time use~\cite{conte_computation, KOHLER20191}. Plug-and-play MPC for linear dynamics addresses the fixed network issue, supporting modularity and request-based PnP operations~\cite{Zeilinger2013plugplay, riverso2012}. Coalitional MPC extends this idea by allowing agents to merge or split coalitions to improve performance, but is typically restricted to linear settings and requires nontrivial coordination overhead~\cite{masero2023robust,maxim2024coalitional}.
For nonlinear dynamics, distributed optimization methods such as Augmented Lagrangian Alternating Direction Inexact Newton (ALADIN)~\cite{ALADIN_based_distributed_MPC_for_dynamic_partitioning} or distributed Sequential Quadratic Programming (SQP)/ADMM~\cite{stomberg2025cooperativedistributedmodelpredictive, carron2023multiagent} have been proposed. While they extend beyond linear settings, they usually rely on a centralized coordinator~\cite{ZanonOptimalCoordinationIntersection} or require many communications/iterations to converge. In addition, these methods are inherently iterative, requiring repeated neighbor exchanges (or sequential, order-dependent updates) within each MPC step to coordinate coupled constraints~\cite{sequential_multiplex_control, sequential_dmpc}.
Application-driven approaches demonstrate partial plug-and-play capabilities: for example, hierarchical nonlinear MPC in buildings~\cite{wang2024demand} supports adding or removing devices and rooms, but depends on centralized layers and application-specific guarantees, whereas distributed nonlinear MPC (NMPC) frameworks~\cite{burk2022modular} enable runtime inclusion/removal of agents in smart grid applications, but without formal guarantees.
To reduce the computational and coordination load of distributed optimization algorithms, non-iterative distributed MPC variants have been explored. These approaches exchange predicted trajectories and adapt constraints~\cite{FARINA20121088}, introduce consistency sets~\cite{wiltz2022consistency}, or apply contracts and reachable sets for coupling~\cite{Lucia2015contract}. However, their guarantees are tied to fixed networks, hard to compute robust invariant sets, and can block agents from joining or leaving to preserve feasibility.
Beyond MPC-centric designs, other coordination schemes have been investigated: priority-based ordering~\cite{Alrifaee_priority}, game-theoretic approaches~\cite{karma, fieni2024gametheoreticenergymanagementstrategies, liniger2019noncooperativegameapproachautonomous}, traffic-specific rules~\cite{AssistingDrivers, BidirectionalMixedTraffic, priority_based_highway_merging}, Voronoi-based collision avoidance~\cite{Rakovic2021, Luis_2020, Soria}, and CBF-based safety certificates~\cite{goarin2024decentralizednonlinearmodelpredictive, Dimos_CBF}. Learning-based approaches have also been widely investigated for multi-agent control. However, since our focus is on methods that provide formal guarantees on safety and feasibility, we do not review this line of work in detail. Finally, safety-verification approaches can process request-based plug-and-play operations, but result in conservative behavior~\cite{carron_plug_and_play_safety_verification,ohnemus2026distributed}. 
In contrast, our contract-based DMPC targets fully distributed, anytime plug-and-play with neighbor-only communication and formal recursive feasibility and safety guarantees. A category-wise comparison, covering assumptions, communication requirements, and guarantees, is deferred to the discussion Section~\ref{sec: Discussion} and summarized in Table~\ref{tab:comparison}.

\subsubsection*{Notation}

Throughout the paper, we denote by $\mathbb{R}$ the set of real numbers, by $\mathbb{N}$ the set of non-negative integers and by $\mathbb{N}_a^b=\{ n \in \mathbb{N} \ | \ a\leq n \leq b \}$. The zero vector in $\mathbb{R}^n$ is denoted by $\mathbf{0}$. The discrete-time index is denoted by $\distime \in \mathbb{N}$, with sampling time $T_s$. We denote by $x_{\step|\distime}$ the value of a variable at time $\distime+\step$ predicted at time $\distime$. More generally, for agent-indexed quantities we write $x_{\firstagent,\step|\distime}$, and for horizon-indexed collections we use the shorthand $\mathcal{P}_{\firstagent,\distime}^{N+1}=\{\mathcal{P}_{\firstagent,0|\distime},\dots,\mathcal{P}_{\firstagent,N|\distime}\}$, or alternatively $\mathcal{P}_{\firstagent,\distime}^{0:N}$ when the starting and ending indices are important.
For $c\in\mathbb{R}^n$ and $\epsilon>0$, we define the open ball of radius $\epsilon$ centered at $c$ as $\mathcal{B}_{\epsilon}(c):=\{x\in\mathbb{R}^n : \|x - c\| < \epsilon \}$. When $c=\mathbf{0}$, we simply write $\mathcal{B}_{\epsilon}$. The sum of two sets $\mathcal{A}$, $\mathcal{B}\subseteq\mathbb{R}^n$, denoted as $\mathcal{A}\oplus\mathcal{B}$, is $\{a+b \ | \ a\in \mathcal{A},\ b\in \mathcal{B}\}$ while  $\mathcal{A}\ominus\mathcal{B}$ is $\{x\in\mathbb{R}^n \ | \ x+b \in \mathcal{A} , \ \forall b \in \mathcal{B} \}$. Finally, given a finite set $\mathcal{N}$, we denote its cardinality by $|\mathcal{N}|$.
When considering finite horizon optimal control problems (FHOCP), the length of the horizon is denoted by $N$ and each step along the horizon by $\step \in \mathbb{N}_0^N$. 


\section{Problem Description}\label{sec:problem_description}
In this section we introduce the system setup, define the communication capabilities among agents, and state the control problem. We further consider the necessary safety-related requirements and resulting problem formulation.

\subsection{System setup}
We consider a finite population of mobile robotic agents indexed by the set $\mathcal 
{M}$.  Each agent $\firstagent \in \mathcal{M}$ is described by nonlinear dynamics
\begin{align}\label{eq:dynamics}
    x_{\firstagent}(\distimenext) = & f_{\firstagent}(x_{\firstagent}(\distime), u_{\firstagent}(\distime)), \\
    p_{\firstagent}(\distime) = & C_{\firstagent} x_{\firstagent}(\distime),
\end{align}
where $x_{\firstagent}\in\mathbb{R}^{n_i}, \ u_{\firstagent}\in\mathbb{R}^{m_i}$ describe the state and input of the system and $f_{\firstagent}:\mathbb{R}^{n_i}\times \mathbb{R}^{m_i}\rightarrow \mathbb{R}^{n_i}$. The position of the agent $p_{\firstagent} \in \mathbb{R}^{3}$ is defined as a subset of the states, where $C_{\firstagent} \in \mathbb{R}^{ 3\times n_i}$ is an appropriate selection matrix. The agent is assumed to have a volume, and the position of the agent refers to its center of mass. The agent can be approximated by a ball, $\mathcal{B}_{\epsilon/2}$ representing the physical encumbrance of the agent.
The system is subject to local state and input constraints for each agent:
\begin{align}
    &x_{\firstagent}(\distime) \in \mathcal{X}_{\firstagent}, \label{eq:stateconstraints}  \\
    &u_{\firstagent}(\distime) \in \mathcal{U}_{\firstagent}, \quad \forall i\in \mathcal{M}, \quad\forall \distime \geq 0. \label{eq:inputconstraints}
\end{align}
Let us consider, without loss of generality, that $p_{\firstagent}(\distime)$ is at the top of the state vector $x_{\firstagent}(\distime)$ and introduce the operator $\phi (p_{\firstagent}) = [p_{\firstagent}^\top,0, \dots, 0]^\top\in\mathbb{R}^{n_{\firstagent}}$ that generates a vector of the dimension of $x_{\firstagent}(\distime)$ with first entry, corresponding to the position $p_{\firstagent}$ of the ${\firstagent}$-th agent.
\begin{assum} \label{ass:posInv}
The agents' dynamics~\eqref{eq:dynamics} are position invariant, i.e. $\forall p_{\secondagent} \in \mathbb{R}^3,\ x_{\firstagent}(\distime),\ u_{\firstagent}(\distime), \ f_{\firstagent}(x_{\firstagent}(\distime)+\phi(p_{\secondagent}),u_{\firstagent}(\distime))=x_{\firstagent}(\distimenext)+\phi(p_{\secondagent})$.
\end{assum}
Furthermore, the system is subject to coupling constraints to enforce collision avoidance among agents,
\begin{align}\label{eq:coupling_constraint}
    h(x_{\firstagent}(\distime), x_{\secondagent}(\distime)) \geq 0, \quad \forall \firstagent,\secondagent \in \mathcal{M}, \quad \firstagent \neq \secondagent, \quad \forall \distime\geq 0.
\end{align}

To streamline the exposition, we assume symmetric agents and that constraint~\eqref{eq:coupling_constraint} can be expressed as
\begin{equation} \label{eq:collision_avoidance}
     h(x_{\firstagent}(\distime), x_{\secondagent}(\distime)) = \| p_{\firstagent}(\distime)-p_{\secondagent}(\distime) \| - \epsilon. 
\end{equation}
\begin{rem}
    Extensions to arbitrary agent geometries can be done by considering a suitable metric that includes the orientation in addition to the position for collision avoidance.
\end{rem}

\subsection{Communication}
\label{sec:communication}
The communication set defines the physical range in space within which an agent can communicate with its neighbors. We assume homogeneous communication, meaning that all agents have the same communication capabilities. This assumption is not strictly required, and a similar strategy to the one presented in this work could be adopted, with suitable modifications, in the heterogeneous case, however this is omitted for clarity. The communication set $\mathcal{R}(p_{\firstagent}) \subset \mathbb{R}^3$ is a non-empty closed convex set centered at the current position of the agent $p_{\firstagent}$.
For simplicity, this set is under-approximated by a centrally symmetric polytope centered at the agent's current position $p_{\firstagent}$,
\begin{align}
    \mathcal{R}(p_{\firstagent}) := \{\xi\in\mathbb{R}^3 : A (\xi-p_\firstagent) \leq b\},
\end{align}
for suitable matrices $A\in\mathbb{R}^{n_c\times 3}$ and $b\in\mathbb{R}^{n_c}$ chosen such that $\mathcal{R}(\mathbf{0}_3)$ is centrally symmetric.
Thus, if the position of an agent $\secondagent \in \mathcal{M}$, lies within the communication set of agent $\firstagent$, then $\secondagent$ is considered a neighbor of $\firstagent$ and bidirectional communication is established. 

As an example, consider the case where two agents can communicate if the Euclidean distance between them is less than or equal to a given radius $r$. 
\begin{equation}
    \| p_{\firstagent} - p_{\secondagent} \| \leq r.
\end{equation}
In this case, the communication set can be represented by any polytopic under-approximation of a ball with radius $r$, centered at each agent's position. A proper design of $\mathcal{R}$ should account for factors such as transmission power, receiver sensitivity, antenna characteristics, operating frequency, environmental conditions, line-of-sight availability, or interference - a detailed definition is application-specific and beyond the scope of this paper. \\
An equivalent representation of the set of neighboring agents can be obtained using the intersection of two sets centered at the current positions of the agents. Specifically, agents 
$\firstagent$ and 
$\secondagent$ are neighbors if the following holds:
\begin{equation} \label{eq:neighbors}
\bar{\mathcal{R}}(p_{\firstagent}) \cap \bar{\mathcal{R}}(p_{\secondagent}) \neq \emptyset,
\end{equation}
where 
$\bar{\mathcal{R}}(p_\firstagent)$ is defined as a tightened version of $\mathcal{R}(p_\firstagent)$, such that:
\begin{equation} \label{eq:comm_set}
\bar{\mathcal{R}}(p_{\firstagent}) := \left\{\xi \in \mathbb{R}^3 : A (\xi - p_{\firstagent}) \leq \frac{b}{2} \right\},
\end{equation}
and we assume that the tightened communication region contains the physical encumbrance of the agent, i.e., $\mathcal{B}_{\epsilon/2} \subseteq \bar{\mathcal{R}}(\mathbf{0}_3)$.
Therefore, the set of neighbors of agent $\firstagent$ at time $\distime$ is
\begin{align}\label{eq:neighbor_set}
\mathcal{N}_{\firstagent}(\distime)
=
\{\secondagent \in \mathcal{M} \setminus \{\firstagent\}: \bar{\mathcal{R}}(p_{\firstagent}(\distime)) \cap \bar{\mathcal{R}}(p_{\secondagent}(\distime)) \neq \emptyset \}.
\end{align}
For ease of readability, the time dependency of $\mathcal{N}_{\firstagent}(\distime)$ is omitted when not crucial.

The neighbor sets induce a time-varying undirected communication graph, where an edge is present between agents $\firstagent$ and $\secondagent$ if and only if $\secondagent\in\mathcal{N}_{\firstagent}(\distime)$. In the considered anytime plug-and-play setting, this graph may change arbitrarily over time as agents move, enter, or leave the communication range. Since the proposed method is formulated entirely in terms of local neighbor sets, no explicit graph-level construction is needed in the remainder of the paper.

\subsection{Problem formulation}
We are now in a position to state the problem we aim to solve.

\begin{prob} \label{pr:mainProb}
Consider a set of agents $\mathcal{M}$, where each agent is governed by the nonlinear dynamics~\eqref{eq:dynamics} and subject to state and input constraints~\eqref{eq:stateconstraints},~\eqref{eq:inputconstraints}. Neighbors are determined according to~\eqref{eq:neighbor_set}, leading to a time-varying plug-and-play communication topology that depends on the agents' positions. We aim to design a distributed control strategy that relies solely on neighbor-to-neighbor communication, ensures collision avoidance among agents~\eqref{eq:coupling_constraint}, and allows each agent to minimize its own local objective.
\end{prob}

At each time-step $\distime$ this objective can be cast as the centralized FHOCP
\begin{subequations}\label{centralized_mpc}
\begin{alignat}{3}
   &\min_{\substack{x_{\firstagent, 0|\distime}, \hdots, x_{\firstagent, N| \distime}, \\
   u_{\firstagent, 0| \distime}, \hdots,u_{\firstagent, N-1|\distime}, \\
   \forall \firstagent \in\mathcal{M}}}
   && \sum_{\firstagent\in\mathcal{M}}\sum_{\step=0}^{N-1}
   J_{\firstagent}(x_{\firstagent, \step|\distime}, u_{\firstagent, \step|\distime})
   \label{centralized_mpc:cost} \\
   &\operatorname{subject~to}
   && x_{\firstagent,0|\distime} = x_{\firstagent}(\distime)
   \label{centralized_mpc:init_constr},\\
   &&
   & x_{\firstagent, \stepnext|\distime}= f_\firstagent(x_{\firstagent, \step|\distime}, u_{\firstagent, \step|\distime})
   \label{centralized_mpc:dynamics_constr}, \\
   &&
   & x_{\firstagent, \step|\distime} \in \mathcal{X}_{\firstagent}
   \label{centralized_mpc:state_constr}, \\
   &&
   & u_{\firstagent, \step|\distime} \in \mathcal{U}_{\firstagent}
   \label{centralized_mpc:input_constr}, \\
   &&
   & h(x_{\firstagent,\step|\distime},x_{\secondagent,\step|\distime}) \geq 0
   \label{centralized_mpc:collision_avoidance_constr}, \\
   &&
   & \forall \firstagent \in\mathcal{M},\ \forall \secondagent \in\mathcal{M}\setminus\{\firstagent\}, \nonumber
\end{alignat}
\end{subequations}
where dynamics~\eqref{centralized_mpc:dynamics_constr} and input constraints~\eqref{centralized_mpc:input_constr} are imposed for all $\step\in\mathbb{N}_0^{N-1}$, while the state constraints~\eqref{centralized_mpc:state_constr} and collision-avoidance constraints~\eqref{centralized_mpc:collision_avoidance_constr} are imposed for all $\step\in\mathbb{N}_0^{N}$.
The first control action of the optimal input sequence is applied in a receding-horizon fashion,
\begin{equation}
    u_{\firstagent}(\distime)=u_{\firstagent,0|\distime}^{*}.
\end{equation}

\begin{rem}
For simplicity, the exposition is restricted to a finite population of agents. This ensures that the centralized problem is well posed and keeps the notation compact. However, since the proposed construction and the corresponding feasibility and safety arguments are local and rely only on neighbor-to-neighbor interactions, the framework extends naturally to countably infinite locally finite populations, i.e., populations satisfying $|\mathcal{N}_{\firstagent}(\distime)|<\infty$ for all $\firstagent\in\mathcal{M}$ and $\distime\in\mathbb{N}$.
\end{rem}

\begin{figure*}[t!]
    \centering
    \includegraphics[trim={1.2cm 4.5cm 2.2cm 4.5cm},clip,width=1.0\linewidth]{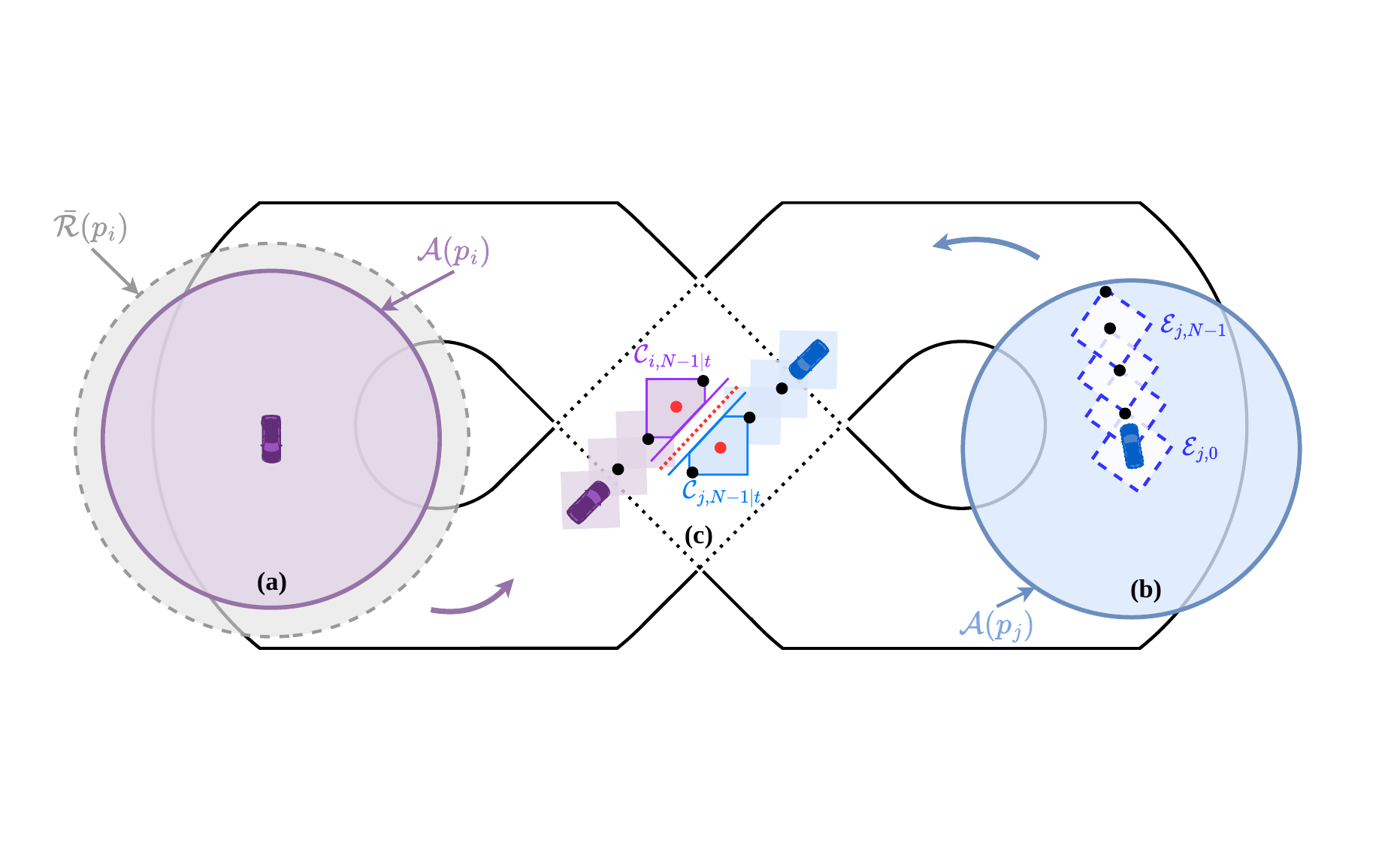}
    \caption{Consider two neighboring agents $\firstagent$ (purple car)
    and $\secondagent$ (blue car).
    \textbf{(a)} Visualization of the \textbf{communication set} $\bar{\mathcal{R}}(p_{\firstagent})$ (\filledcircle[gray, fill=gray!60]{3pt} gray circle)
    and \textbf{awareness set} $\mathcal{A}(p_{\firstagent})$ (\filledcircle[SetPurple, fill=SetPurple!70]{3pt} purple circle) for agent $\firstagent$. \textbf{(b)} Visualization of the \textbf{safety envelopes} $\mathcal{E}_{\secondagent, \step}, \forall \step \in \mathbb{N}_{0}^{N-1}$ (\filledsquare[fill=white,draw=ContractBlue]{1.5ex}blue dashed squares), 
    which constrain the one-step prediction, thereby resulting in the edge of the awareness set not being reached in less than $N$ steps. 
    \textbf{(c)} Visualization of the \textbf{cells} for agents $\firstagent$ and $\secondagent$ at stage $N-1$ along the horizon (\textcolor[HTML]{007FFF}{\rule[1pt]{8pt}{1pt}} blue and \textcolor[HTML]{9933FF}{\rule[1pt]{8pt}{1pt}} purple lines respectively) and the resulting \textbf{contracts} $\mathcal{C}_{\firstagent,N-1|t}$ and $\mathcal{C}_{\secondagent, N-1|t}$ (\filledsquare[fill=SetPurple!60,draw=ContractPurple]{1.5ex}purple and \filledsquare[fill=SetBlue!60,draw=ContractBlue]{1.5ex}blue cut-off squares).
    }
    \label{fig:constraint_overview}
    \vspace{-5mm}
 \end{figure*}
\section{Anytime Plug-and-Play MPC} \label{sec:method}

To solve Problem~\ref{pr:mainProb}, we propose a distributed receding-horizon strategy that can be solved locally by each agent using only neighbor-to-neighbor communication. While the cost~\eqref{centralized_mpc:cost} and the local state/input constraints~\eqref{centralized_mpc:state_constr}--\eqref{centralized_mpc:input_constr} are separable across agents, the collision-avoidance constraint~\eqref{centralized_mpc:collision_avoidance_constr} couples the optimization variables of different agents. Our goal is therefore to replace this coupled constraint by local \emph{contracts} that can be constructed from one exchange of predicted trajectories with the current neighbors and will be detailed in the following.

At each time $\distime$, the neighbor set $\mathcal{N}_{\firstagent}(\distime)$ is computed from the current positions and kept fixed while solving the FHOCP at time $\distime$. Based on this frozen neighbor set, agent $\firstagent$ builds a \emph{contract}
\begin{equation}\label{eq:contract}
    \mathcal{C}_{\firstagent,\distime}
    :=
    \big(\mathcal{P}_{\firstagent,\distime}^{N+1},\,\mathcal{E}_{\firstagent}^{N}\big),
\end{equation}
where $\mathcal{P}_{\firstagent,\distime}^{N+1}$ is a collection of time-varying \emph{cells} used to separate the current neighbors, and $\mathcal{E}_{\firstagent}^{N}$ is a collection of \emph{safety envelopes} used to preserve recursive feasibility and robustness with respect to topology changes. These two ingredients allow one to derive a distributed contract-based local FHOCP, inspired by the centralized reference problem~\eqref{centralized_mpc}, with only one exchange of predicted trajectories per time step. Figure~\ref{fig:constraint_overview} summarizes this construction, from the communication and awareness sets to the cells, safety envelopes, and resulting local contracts.

\subsection{Contract ingredients}

\subsubsection{Cells}
\begin{figure}
    \centering
    \includegraphics[width=1.0\linewidth]{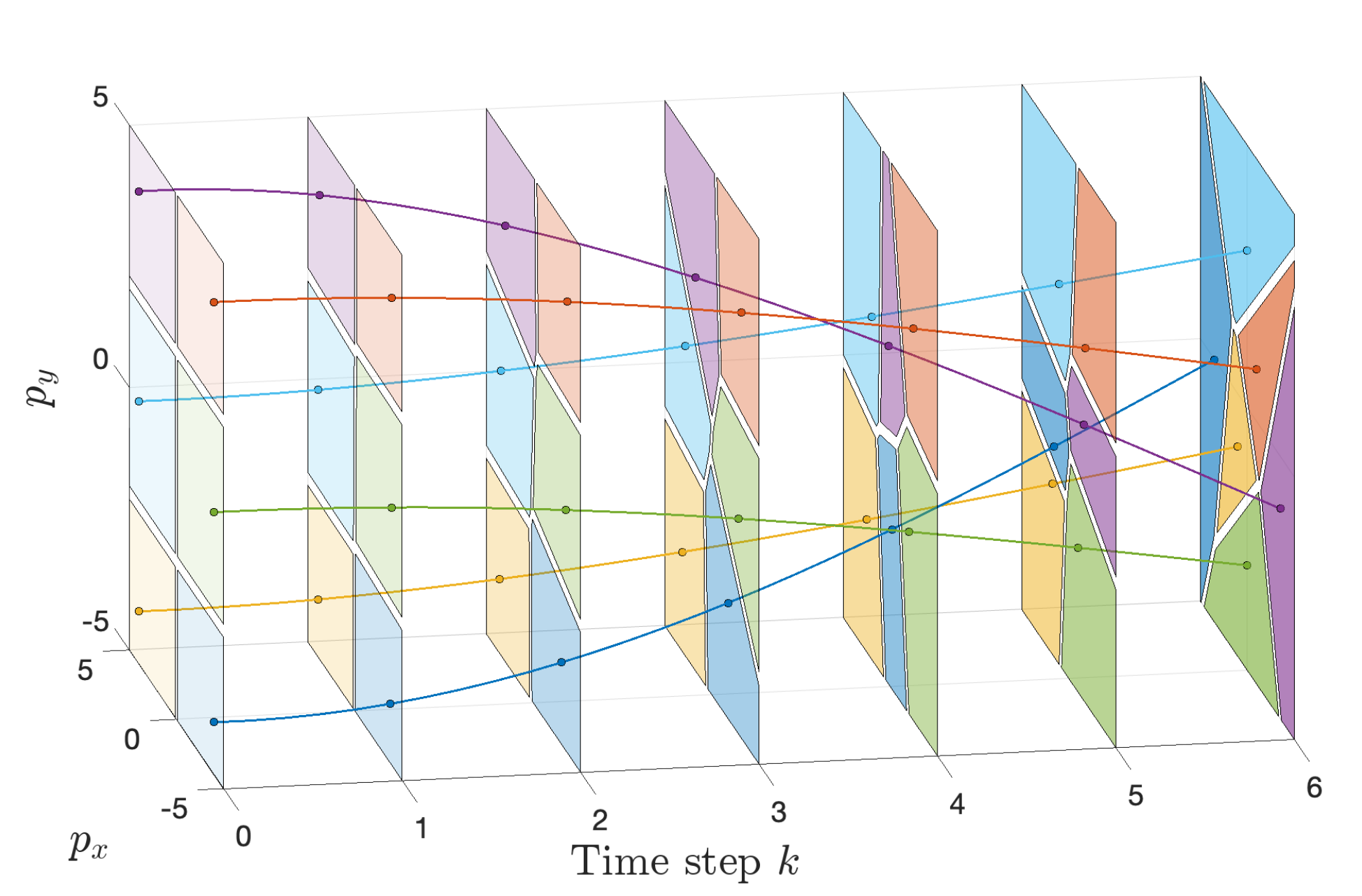}
    \caption{\textbf{Cells:} Visualization of partitioning, consisting of cells, of multiple agents over the horizon. Each agent is associated with a different color. The colored points indicate the position of an agent at a given step $k$ along the horizon. The respective hyperplanes are given as the shaded polytopes in the same color. The cells are constructed via a Voronoi-type partition; see Section~\ref{sec:application_to_racing}. Therefore, they change as the predicted positions of the agents change.}
    \label{fig:hyperplanes_as_partitions}
    \vspace{-5mm}
\end{figure}
Cells are stage-wise convex regions computed from the predicted trajectories exchanged with the \emph{current} neighbors. Their role is to ensure collision-free trajectories between neighbors.
Collision avoidance with agents that are not current neighbors will be handled through awareness sets in the next subsection. Figure~\ref{fig:hyperplanes_as_partitions} illustrates a possible time-varying cell partition along the prediction horizon.
\begin{definition}\label{def:partitions}
A \emph{cell} for agent $\firstagent$ at prediction step $\step$ and time $\distime$ is a convex subset of the admissible position space
\begin{align}
    \mathcal{P}_{\firstagent,\step|\distime}
    \subseteq
    \{ p_\firstagent=C_\firstagent x_\firstagent \;|\; x_\firstagent \in \mathcal{X}_\firstagent \}.  
    \label{eq:partition_state_subset}
\end{align}
The set of cells assigned to agent $\firstagent$ over a horizon of length $N$ is 
\begin{align}
    \mathcal{P}^{N+1}_{\firstagent,\distime}
    =
    \{\mathcal{P}_{\firstagent,0|\distime}, \dots, \mathcal{P}_{\firstagent,N|\distime}\}.
    \label{eq:set_of_partition}
\end{align}
A collection of cells $\{\mathcal{P}^{N+1}_{\firstagent,\distime}\}_{\firstagent\in\mathcal{M}}$ is said to be \emph{admissible} at time $\distime$ if, for all $\firstagent\in\mathcal{M}$, all $\secondagent\in\mathcal{N}_{\firstagent}(\distime)$, and all $\step\in\mathbb{N}_0^{N}$,
\begin{align}
    (\mathcal{P}_{\firstagent,\step|\distime}\oplus\mathcal{B}_{\epsilon/2})
    \;\cap\;
    (\mathcal{P}_{\secondagent,\step|\distime}\oplus\mathcal{B}_{\epsilon/2})
    = \emptyset.
    \label{eq:compatible_cells}
\end{align}
\end{definition}

Each local control problem then enforces the cell-membership constraint
\begin{align}
    p_{\firstagent,\step|\distime}\in \mathcal{P}_{\firstagent,\step|\distime},
    \qquad
    \forall \step\in\mathbb{N}_0^{N}.
    \label{eq:partition_constraint}
\end{align}

The next proposition states the basic role of cells: if neighboring agents remain inside cells, then they remain collision-free.

\begin{prop}\label{prop:partitions}
Suppose that, at time $\distime$, each agent $\firstagent\in\mathcal{M}$ is assigned a collection of cells $\mathcal{P}^{N+1}_{\firstagent,\distime}$ such that the family $\{\mathcal{P}^{N+1}_{\firstagent,\distime}\}_{\firstagent\in\mathcal{M}}$ is admissible in the sense of Definition~\ref{def:partitions}, and that the predicted positions satisfy~\eqref{eq:partition_constraint}. Then, for every $\firstagent\in\mathcal{M}$, every $\secondagent\in\mathcal{N}_{\firstagent}(\distime)$, and every $\step\in\mathbb{N}_0^{N}$,
\begin{align}
    h(x_{\firstagent,\step|\distime},x_{\secondagent,\step|\distime}) \ge 0.
\end{align}
\end{prop}
\noindent The proof is given in the Appendix~\ref{App:proof_partitions}.

The following converse result shows that, whenever collision-free trajectories are given, there exists an admissible collection of cells containing them.

\begin{lemma}\label{lem:cell_existence}
Consider a time instant $\distime$ and predicted positions
$p_{\firstagent,\step|\distime}\in \{C_\firstagent x_\firstagent \mid x_\firstagent\in\mathcal X_\firstagent\}$,
for all $\firstagent\in\mathcal M$ and $\step\in\mathbb N_0^{N}$.
Assume that, for every $\firstagent\in\mathcal M$, every $\secondagent\in\mathcal N_{\firstagent}(\distime)$, and every $\step\in\mathbb N_0^{N}$,
\begin{align}
    h(x_{\firstagent,\step|\distime},x_{\secondagent,\step|\distime}) \ge 0.
\end{align}
Then there exists a family of cells
$\{\mathcal P^{N+1}_{\firstagent,\distime}\}_{\firstagent\in\mathcal M}$
admissible in the sense of Definition~\ref{def:partitions} such that
\begin{align}
    p_{\firstagent,\step|\distime}\in \mathcal P_{\firstagent,\step|\distime},
    \qquad
    \forall \firstagent\in\mathcal M,\ \forall \step\in\mathbb N_0^{N}.
\end{align}
\end{lemma}
\noindent The proof can be found in Appendix~\ref{App:proof_lem_partitions}.

\subsubsection{Awareness sets and safety envelopes}

Cells are sufficient to separate the current neighbors, but by themselves they do not address topology changes. To this end, we introduce an auxiliary set that captures the region in which an agent may become aware of other agents.

\begin{definition}\label{definition:safe_set}
The \emph{awareness set} of agent $\firstagent$ is
\begin{align}
    \mathcal{A}(p_{\firstagent})
    :=
    \bar{\mathcal{R}}(p_{\firstagent}) \ominus \mathcal{B}_{\epsilon/2},
\end{align}
where $\mathcal{B}_{\epsilon/2}$ accounts for the physical size of the agent and is centered at the origin.
\end{definition}

The awareness set is a tightened version of the communication set~$\bar{\mathcal{R}}(p_{\firstagent})$ that accounts for the physical size of the agent. Since it is obtained by subtracting the body shape centered at the origin from the translated communication set, it remains centered at the current position of the agent. The awareness set represents the region around agent $i$ in which another agent can be detected sufficiently early to become a neighbor before a collision can occur. Since $A(p_i)$ is a tightened version of the communication set, any agent inside $A(p_i)$ is necessarily within the communication range of agent $i$.
We next describe a pairwise safety property for agents that are not neighbors at the current time.

\begin{prop}\label{prop:non_neighbor_safe_set}
Consider two distinct agents $\firstagent,\secondagent\in\mathcal{M}$ and a time instant $\distime$ such that $\secondagent \notin \mathcal{N}_{\firstagent}(\distime)$.
If the predicted trajectories satisfy
\begin{align}
    p_{\firstagent,\step|\distime}\in \mathcal{A}(p_{\firstagent,0|\distime}),
    \quad
    p_{\secondagent,\step|\distime}\in \mathcal{A}(p_{\secondagent,0|\distime}),
    \quad
    \forall \step\in\mathbb{N}_0^{N},
    \label{eq:pairwise_safe_set_constraint}
\end{align}
then $h(x_{\firstagent,\step|\distime},x_{\secondagent,\step|\distime}) \ge 0,$ $\forall \step\in\mathbb{N}_0^{N}.$
\end{prop}
\noindent The proof is given in the Appendix~\ref{App:proof_cells1}.

While~\eqref{eq:pairwise_safe_set_constraint} guarantees pairwise collision avoidance at the current time for non-neighboring agents, it is in general not preserved after shifting the solution by one time step. The reason is that the awareness set is centered at the current position and therefore shifts from one time step to the next. As a consequence, the tail of a previously feasible trajectory may leave the shifted awareness set. This issue is illustrated in Figure~\ref{fig:safe_set_loss_rec_feas}.

\begin{figure}[t]
    \centering       \includegraphics[width=0.75\linewidth, trim={1.2cm 0.5cm 1.2cm 0.2cm},clip]{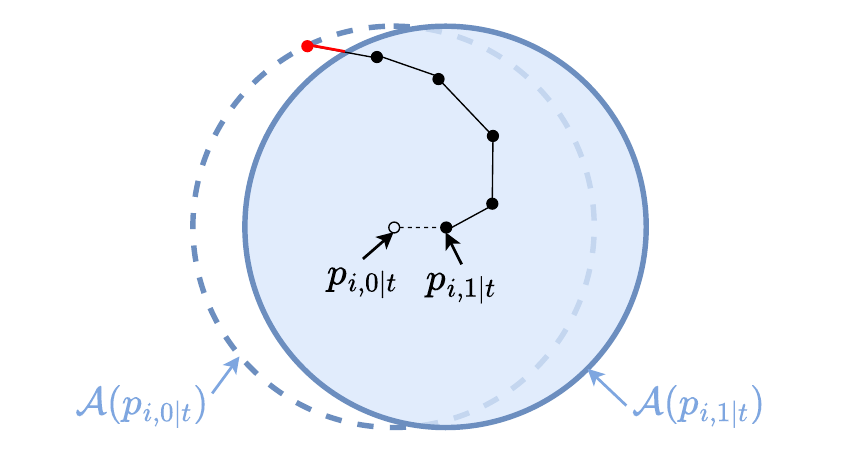}
  \caption{Visualization of the loss of feasibility when the shifted solution is used as a candidate trajectory. The current position of the agent at time $\distime$ is shown by the black circle $p_{\firstagent,0|\distime}$. The planned trajectory is shown in black, and the awareness set $\mathcal{A}(p_{\firstagent,0|\distime})$ at time $\distime$ is shown as the blue dashed circle. At the next time step $\distimenext$, the solution is shifted, and the awareness set is centered at the shifted current position $p_{\firstagent,1|\distime}$. The tail of the shifted trajectory is no longer contained in the shifted awareness set (red) and is therefore not feasible.
}
  \label{fig:safe_set_loss_rec_feas} 
\vspace{-3mm}
\end{figure}

To ensure that feasibility is preserved after shifting the solution by one time step, we do not constrain the whole trajectory to remain in a fixed awareness set. Instead, we constrain the successive increments of the predicted positions using \textit{safety envelopes}, which are scaled-down copies of the awareness set. Their cumulative effect guarantees that the predicted trajectory remains inside the awareness set.

\begin{figure*}[t!]
    \centering
    \includegraphics[trim={0.5cm 0.2cm 1cm 0.2cm},clip,width=1.0\linewidth]{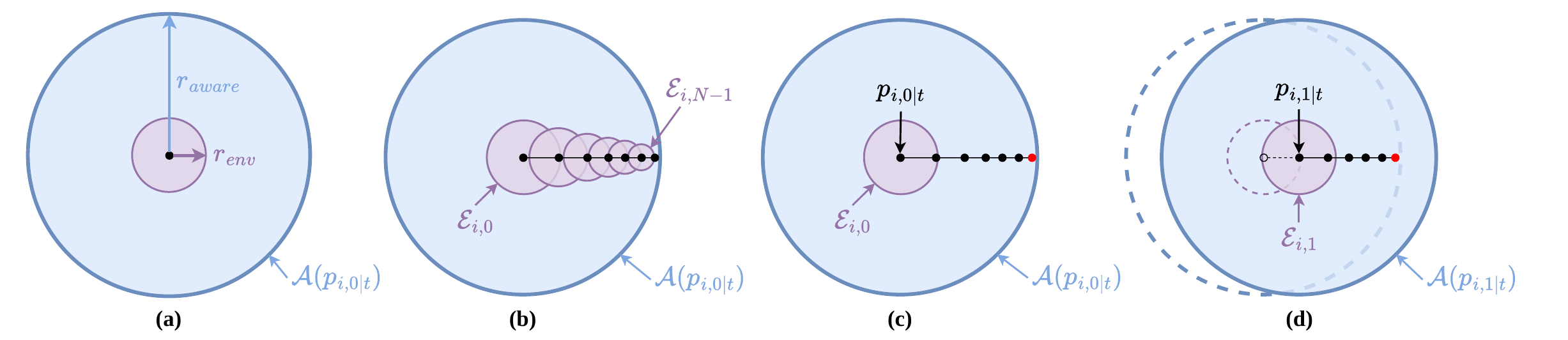}
    \caption{Example of a set of \textbf{safety envelopes} that fulfill Definition~\ref{def:safety_envelope}. \textbf{(a)} The safety envelopes (\filledcircle[SetPurple, fill=SetPurple!60]{3pt} smaller purple circles)
    are a scaled-down version of the awareness set $\mathcal{A}(p_{\firstagent,0| \distime})$ (\filledcircle[SetBlue, fill=SetBlue!60]{3pt} large blue circle)
    , $r_{\textit{env}} = \alpha_{\step} \cdot r_{\textit{aware}}$, and \textbf{(b)} the scaling factors sum up to 1. The safety envelopes ($\mathcal{E}_{\firstagent,\step}, \forall \step \in \mathbb{N}_0^{N-1}$) ensure the trajectory $\tau_{\firstagent, \distime}^{0:N}$ (\filledcircle[black, fill=black]{1pt} sequence of black points) does not reach the edge of the awareness set $\mathcal{A}(p_{\firstagent, 0|\distime})$ in less than $N$ steps. \textbf{(c)} Visualization of how the \textit{one-step} constraint results in feasible shifted solution. The current position of the agent at time $\distime$ is shown as the black point indicated by $p_{\firstagent, 0|\distime}$. The planned trajectory is given in black and the awareness set at time $\distime$ as the blue shaded area. The safety envelope for the first step is shown in purple. \textbf{(d)} At the next time step $\distimenext$, the solution is shifted, and the awareness set is now centered around the predicted next step (black point $p_{\firstagent, 1|\distime}$). The shifted trajectory (black) is still contained in the awareness set and, therefore, a feasible solution.
    }
    \label{fig:safety_envelope_overview}
    \vspace{-5mm}
 \end{figure*}

\begin{definition}\label{def:safety_envelope}
A \emph{safety envelope} for agent $\firstagent$ is a sequence of convex sets 
\[
\mathcal{E}_{\firstagent}^{N}
=
\{\mathcal{E}_{\firstagent,0}, \dots, \mathcal{E}_{\firstagent,N-1}\},
\] 
where each set is defined as
\begin{align}
    \mathcal{E}_{\firstagent,\step}
    =
    \alpha_{\firstagent,\step}\,\mathcal{A}(\mathbf{0}_3),
    \quad
    \alpha_{\firstagent,\step}\geq 0,
    \quad
    \forall \step\in\mathbb{N}_{0}^{N-1}.
\end{align}
The coefficients $\alpha_{\firstagent,\step}$ are chosen such that
\begin{align}
    \alpha_{\firstagent,\step}\geq \alpha_{\firstagent,\step+1},
    \qquad
    \sum_{\step=0}^{N-1}\alpha_{\firstagent,\step}\leq 1.
    \label{eq:alpha_decrease}
\end{align}
Consequently,
\begin{align}
    \bigoplus_{\step=0}^{N-1}\mathcal{E}_{\firstagent,\step}
    \subseteq
    \mathcal{A}(\mathbf{0}_3).
    \label{eq:sum_of_safety_envelopes}
\end{align}
\end{definition}
\noindent
The safety-envelope constraint is then imposed as
\begin{align}
    p_{\firstagent,\stepnext|\distime} - p_{\firstagent,\step|\distime}
    \in \mathcal{E}_{\firstagent,\step},
    \qquad
    \forall \step\in\mathbb{N}_0^{N-1}.
    \label{eq:constr_safety_envelope}
\end{align}
The next proposition states the key implication of the safety-envelope constraint: if the successive position increments satisfy~\eqref{eq:constr_safety_envelope}, then the whole predicted trajectory, and every shifted tail of it, remains inside the corresponding awareness set.
\begin{prop}\label{prop:safety_envelopes}
Consider an agent $\firstagent$ and its predicted position trajectory
\[
\tau_{\firstagent,\distime}^{0:N}
=
\{p_{\firstagent,0|\distime},\dots,p_{\firstagent,N|\distime}\}.
\]
If~\eqref{eq:constr_safety_envelope} holds, then the trajectory is contained in the awareness set centered at its initial point, i.e.,
\begin{align}
    \tau_{\firstagent,\distime}^{0:N}
    \subseteq
    \mathcal{A}(p_{\firstagent,0|\distime}),
    \label{eq:traj_in_safe}
\end{align}
and, more generally, for every $\ell\in\mathbb{N}_0^{N-1}$,
\begin{align}
    \tau_{\firstagent,\distime}^{\ell:N}
    :=
    \{p_{\firstagent,\ell|\distime},\dots,p_{\firstagent,N|\distime}\}
    \subseteq
    \mathcal{A}(p_{\firstagent,\ell|\distime}).
    \label{eq:shift_traj_in_safe}
\end{align}
\end{prop}
\noindent The proof can be found in Appendix~\ref{App:proof_cells2}.

Figure~\ref{fig:safety_envelope_overview} illustrates how~\eqref{eq:constr_safety_envelope} ensures that the trajectory $\tau_{\firstagent,\distime}^{0:N}$ remains inside the awareness set while preserving a feasible shifted solution.
\\
The final FHOCP formulation as well as a theoretical analysis are provided in the following section.

\subsection{Local formulation and theoretical analysis} \label{sec:analysis}

Using the local contract~\eqref{eq:contract}, we derive one local contract-based FHOCP per agent, inspired by the centralized reference problem~\eqref{centralized_mpc}. For a fixed agent $\firstagent$, solve
\begin{subequations}\label{distributed_mpc_final}
\begin{alignat}{2}
   &\min_{\substack{x_{\firstagent,0|\distime},\dots,x_{\firstagent,N|\distime},\\ \,u_{\firstagent,0|\distime},\dots,u_{\firstagent,N-1|\distime},\\ \,\bar x_{\firstagent|\distime}, \bar u_{\firstagent|\distime}}}
   && \sum_{\step=0}^{N-1}
   J_{\firstagent}(x_{\firstagent,\step|\distime},u_{\firstagent,\step|\distime})
   \label{distributed_mpc_final:cost} \\
   &\text{subject to} 
   && x_{\firstagent,0|\distime} = x_{\firstagent}(\distime)
   \label{distributed_mpc_final:init_constr},\\
   &&
   & x_{\firstagent,\stepnext|\distime}
   = f_\firstagent(x_{\firstagent,\step|\distime},u_{\firstagent,\step|\distime})
   \label{distributed_mpc_final:dynamics_constr}, \\
   &&
   & p_{\firstagent,\step|\distime}
    =
    C_{\firstagent} x_{\firstagent,\step|\distime},
   \label{distributed_mpc_final:position_def} \\
   &&
   & (x_{\firstagent,\step|\distime}, u_{\firstagent,\step|\distime})\in\mathcal{X}_{\firstagent} \times \mathcal{U}_{\firstagent}
   \label{distributed_mpc_final:stateinput_constr}, \\
   &&
   & p_{\firstagent,\step|\distime}\in\mathcal{P}_{\firstagent,\step|\distime}
   \label{distributed_mpc_final:hyperplane}, \\
   &&
   & \Delta p_{\firstagent,\step|\distime}\in\mathcal{E}_{\firstagent,\step}
   \label{distributed_mpc_final:safety_envelope}, \\
   &&
   & x_{\firstagent,N|\distime} = \bar{x}_{\firstagent|\distime}
   \label{distributed_mpc_final:terminal},  \\
   &&
   & \bar{x}_{\firstagent|\distime}
   =
   f_\firstagent(\bar{x}_{\firstagent|\distime},\bar{u}_{\firstagent|\distime})
   \label{distributed_mpc_final:terminal_eq},  \\
   &&
   & (\bar{x}_{\firstagent|\distime},\bar{u}_{\firstagent|\distime})
   \in
   \mathcal{X}_{\firstagent}\times\mathcal{U}_{\firstagent}
   \label{distributed_mpc_final:terminal_adm}
\end{alignat}
\end{subequations}
The constraints~\eqref{distributed_mpc_final:dynamics_constr},
\eqref{distributed_mpc_final:stateinput_constr}, and
\eqref{distributed_mpc_final:safety_envelope}
are imposed for all $\step\in\mathbb{N}_0^{N-1}$, while
\eqref{distributed_mpc_final:position_def} and
\eqref{distributed_mpc_final:hyperplane}
are imposed for all $\step\in\mathbb{N}_0^{N}$.
Moreover,
$\Delta p_{\firstagent,\step|\distime}
:=
p_{\firstagent,\stepnext|\distime}-p_{\firstagent,\step|\distime}$,
for all $\step\in\mathbb{N}_0^{N-1}$.

The terminal equality constraint~\eqref{distributed_mpc_final:terminal} is imposed to an admissible equilibrium selected online by the optimizer. Because of Assumption~\ref{ass:posInv}, this equilibrium is not tied to a fixed spatial location: the optimizer may select any admissible equilibrium compatible with the local constraints, the terminal cell-membership constraint already contained in~\eqref{distributed_mpc_final:hyperplane} for $\step=N$, and the terminal increment constraint contained in~\eqref{distributed_mpc_final:safety_envelope} for $\step=N-1$.

\begin{rem}\label{rem:terminal_set}
We adopt a terminal equality constraint because it yields a simple and transparent recursive-feasibility argument in the anytime plug-and-play setting. A more general formulation would replace~\eqref{distributed_mpc_final:terminal} by a time-varying terminal set and a local terminal controller. This would allow more flexibility, but such a terminal set would need to remain compatible with the time-varying cells, the safety envelopes, and the changing neighbor sets, making its computation significantly more involved. We leave this extension to future work.
\end{rem}

\begin{theorem}\label{thm:recursive_feasibility}
Assume that problem~\eqref{distributed_mpc_final} is feasible at time $\distime=0$, and that at each time $\distime$ the cells $\mathcal{P}_{\firstagent,\distime}^{N+1}$ satisfy Definition~\ref{def:partitions}, while the safety envelopes $\mathcal{E}_{\firstagent}^{N}$ satisfy Definition~\ref{def:safety_envelope}. Then problem~\eqref{distributed_mpc_final} is recursively feasible. Moreover, the resulting closed-loop system satisfies the local state and input constraints and the collision-avoidance constraint~\eqref{eq:collision_avoidance} for all $\distime\geq 0$, even under anytime plug-and-play changes of the network topology.
\end{theorem}

\begin{proof}
We prove recursive feasibility by constructing, from any feasible optimal solution of~\eqref{distributed_mpc_final} at time $\distime$, a feasible candidate solution at time $\distimenext$. Let 
\begin{align*}
X_{\firstagent|\distime}^{*}
&=
\{x_{\firstagent,0|\distime}^{*},\dots,x_{\firstagent,N|\distime}^{*}\}, \\
U_{\firstagent|\distime}^{*}
&=
\{u_{\firstagent,0|\distime}^{*},\dots,u_{\firstagent,N-1|\distime}^{*}\},
\end{align*}
be the optimal state and input sequences at time $\distime$, and let
$(\bar x_{\firstagent|\distime}^{*},\bar u_{\firstagent|\distime}^{*})$
be the corresponding terminal equilibrium. Define the candidate solution at time $\distimenext$ as
\begin{align*}
    \hat x_{\firstagent,\step|\distimenext}
    &:=
    x_{\firstagent,\step+1|\distime}^{*},
    \quad
    \step\in\mathbb{N}_0^{N-1}, \quad
    \hat x_{\firstagent,N|\distimenext}
    :=
    \bar x_{\firstagent|\distime}^{*},\\
    \hat u_{\firstagent,\step|\distimenext}
    &:=
    u_{\firstagent,\step+1|\distime}^{*},
    \quad
    \step\in\mathbb{N}_0^{N-2}, \quad
    \hat u_{\firstagent,N-1|\distimenext}
    :=
    \bar u_{\firstagent|\distime}^{*},\\
    \hat{\bar x}_{\firstagent|\distimenext}
    &:=
    \bar x_{\firstagent|\distime}^{*},
    \qquad
    \hat{\bar u}_{\firstagent|\distimenext}
    :=
    \bar u_{\firstagent|\distime}^{*}.
\end{align*}
The initial condition is satisfied because
\[
\hat x_{\firstagent,0|\distimenext}
=
x_{\firstagent,1|\distime}^{*}
=
x_{\firstagent}(\distimenext),
\]
and the dynamics hold by construction together with~\eqref{distributed_mpc_final:terminal_eq}. Since the solution at time $\distime$ is feasible, the local state/input constraints and the terminal constraints
\eqref{distributed_mpc_final:stateinput_constr}--\eqref{distributed_mpc_final:terminal_adm}
are also satisfied by the candidate solution.
\\
It remains to verify the contract constraints.

\emph{Safety envelopes:}
for every $\step\in\mathbb{N}_0^{N-2}$,
\[
\hat p_{\firstagent,\step+1|\distimenext}
-
\hat p_{\firstagent,\step|\distimenext}
=
p_{\firstagent,\step+2|\distime}^{*}
-
p_{\firstagent,\step+1|\distime}^{*}
\in
\mathcal{E}_{\firstagent,\step+1}
\subseteq
\mathcal{E}_{\firstagent,\step},
\]
where the last inclusion follows from Definition~\ref{def:safety_envelope}. For the last increment,
\[
\hat p_{\firstagent,N|\distimenext}
-
\hat p_{\firstagent,N-1|\distimenext}
=
\bar p_{\firstagent|\distime}^{*}
-
p_{\firstagent,N|\distime}^{*}
=
\mathbf{0}_3
\in
\mathcal{E}_{\firstagent,N-1},
\]
where the equality follows from~\eqref{distributed_mpc_final:terminal}. Hence~\eqref{distributed_mpc_final:safety_envelope} holds at time $\distimenext$.

Moreover, since the optimal trajectory at time $\distime$ is feasible, Proposition~\ref{prop:safety_envelopes} yields
\[
p_{\firstagent,\step|\distime}^{*}\in \mathcal{A}(p_{\firstagent,0|\distime}^{*}),
\qquad
\forall \step\in\mathbb{N}_0^{N}.
\]
This fact will be used below for newly appearing neighbors.

\emph{Cells:}
Fix an agent $\firstagent$ and partition the neighbors at time $\distimenext$ into
\begin{align*}
    \mathcal{N}^{0}_{\firstagent}
    &:=
    \mathcal{N}_{\firstagent}(\distime)\cap \mathcal{N}_{\firstagent}(\distimenext),
\quad
\mathcal{N}^{-}_{\firstagent}
    :=
    \mathcal{N}_{\firstagent}(\distime)\setminus \mathcal{N}_{\firstagent}(\distimenext),
\\
    \mathcal{N}^{+}_{\firstagent}
    &:=
    \mathcal{N}_{\firstagent}(\distimenext)\setminus \mathcal{N}_{\firstagent}(\distime).
\end{align*}

Consider first $\secondagent\in\mathcal{N}^{0}_{\firstagent}$, i.e., neighbors that remain neighbors. For every $\step\in\mathbb{N}_0^{N-1}$,
\begin{align*}
    \hat p_{\firstagent,\step|\distimenext}
&=
p_{\firstagent,\step+1|\distime}^{*}
\in
\mathcal{P}_{\firstagent,\step+1|\distime},
\\
\hat p_{\secondagent,\step|\distimenext}
&=
p_{\secondagent,\step+1|\distime}^{*}
\in
\mathcal{P}_{\secondagent,\step+1|\distime},
\end{align*}
and the compatibility of the cells at time $\distime$ implies that these shifted positions are collision-free. At the last stage,
\begin{align*}
    \hat p_{\firstagent,N|\distimenext}
&=
p_{\firstagent,N|\distime}^{*}
=
\bar p_{\firstagent|\distime}^{*}
\in
\mathcal{P}_{\firstagent,N|\distime},
\\
\hat p_{\secondagent,N|\distimenext}
&=
p_{\secondagent,N|\distime}^{*}
=
\bar p_{\secondagent|\distime}^{*}
\in
\mathcal{P}_{\secondagent,N|\distime},
\end{align*}
again by~\eqref{distributed_mpc_final:terminal}. Hence the shifted candidate remains separated from all neighbors in $\mathcal{N}^{0}_{\firstagent}$.

For every $\secondagent\in\mathcal{N}^{-}_{\firstagent}$, the corresponding pairwise cell constraint is removed at time $\distimenext$, so no additional argument is needed.

It remains to consider $\secondagent\in\mathcal{N}^{+}_{\firstagent}$, i.e., newly appearing neighbors. Since $\secondagent \notin \mathcal{N}_{\firstagent}(\distime)$ and since the optimal trajectory at time $\distime$ is feasible, Proposition~\ref{prop:safety_envelopes} gives
\begin{align*}
    \hat p_{\firstagent,\step|\distimenext}
    &=
    p_{\firstagent,\step+1|\distime}^{*}
    \in
    \mathcal{A}(p_{\firstagent,0|\distime}^{*}),
    \qquad \forall \step\in\mathbb{N}_0^{N-1},
\\
    \hat p_{\secondagent,\step|\distimenext}
    &=
    p_{\secondagent,\step+1|\distime}^{*}
    \in
    \mathcal{A}(p_{\secondagent,0|\distime}^{*}),
    \qquad \forall \step\in\mathbb{N}_0^{N-1}.
\end{align*}
Moreover, by~\eqref{distributed_mpc_final:terminal},
\begin{align*}
    \hat p_{\firstagent,N|\distimenext}
    &=
    \bar p_{\firstagent|\distime}^{*}
    =
    p_{\firstagent,N|\distime}^{*}
    \in
    \mathcal{A}(p_{\firstagent,0|\distime}^{*}),
\\
    \hat p_{\secondagent,N|\distimenext}
    &=
    \bar p_{\secondagent|\distime}^{*}
    =
    p_{\secondagent,N|\distime}^{*}
    \in
    \mathcal{A}(p_{\secondagent,0|\distime}^{*}),
\end{align*}
where the last inclusion again follows from Proposition~\ref{prop:safety_envelopes} applied to the feasible optimal trajectories at time $\distime$.
Since $\secondagent \notin \mathcal{N}_{\firstagent}(\distime)$, it follows from the definition of the neighbor set and of the awareness sets that $\mathcal{A}(p_{\firstagent,0|\distime}^{*})
\cap
\mathcal{A}(p_{\secondagent,0|\distime}^{*})
=
\emptyset$.
Hence, the shifted candidate positions belong stage-wise to two disjoint awareness sets. Therefore, the shifted candidate trajectories satisfy the assumptions of Proposition~\ref{prop:non_neighbor_safe_set}, with awareness sets centered at the positions at time $\distime$. By Proposition~\ref{prop:non_neighbor_safe_set}, the corresponding candidate positions are collision-free for all $k\in\mathbb{N}_0^{N}$.

Thus, for every pair of agents that are neighbors at time $\distimenext$, the corresponding candidate positions are collision-free for all $\step\in\mathbb{N}_0^{N}$. Therefore, applying Lemma~\ref{lem:cell_existence} to the full family of candidate positions at time $\distimenext$ yields the existence of a compatible family of cells $\{\mathcal{P}^{N+1}_{\firstagent,\distimenext}\}_{\firstagent\in\mathcal{M}}$ such that $\hat p_{\firstagent,\step|\distimenext}\in \mathcal{P}_{\firstagent,\step|\distimenext}$, $\forall \firstagent\in\mathcal{M},\ \forall \step\in\mathbb{N}_0^{N}$.
Hence, the candidate solution satisfies~\eqref{distributed_mpc_final:hyperplane} at time $\distimenext$.

We have therefore shown that the candidate solution is feasible for~\eqref{distributed_mpc_final} at time $\distimenext$, and recursive feasibility follows by induction from feasibility at $\distime=0$.

Since the applied input is always the first element of a feasible solution, the local state and input constraints hold for all $\distime\ge 0$. Moreover, collision avoidance follows from Proposition~\ref{prop:partitions} for current neighbors and from Proposition~\ref{prop:non_neighbor_safe_set} together with Proposition~\ref{prop:safety_envelopes} for agents that are not neighbors at the current time. Thus,~\eqref{eq:collision_avoidance} holds for all $\distime\ge 0$.
\end{proof}

\subsection{Algorithm}
\label{sec:algorithm}

The full online procedure executed by each agent is summarized in Algorithm~\ref{alg:safety_envelope}. We assume a feasible initialization at time $\distime=0$, i.e., each agent has a feasible local solution and an associated planned position trajectory; for example, all agents may initially be at standstill in different locations. In the considered application, the information exchanged among neighbors is the planned position trajectory
$\tau_{\firstagent,\distime}^{0:N}
:=
\{p_{\firstagent,0|\distime},\dots,p_{\firstagent,N|\distime}\}$,
which is sufficient to construct the local cells, since the collision-avoidance contracts are defined in the position space. At each time step, agent $\firstagent$ receives from its current neighbors the trajectories planned at the previous step, computes the cells accordingly, solves its local FHOCP, applies the first control input, and stores the new planned trajectory for the next communication round. A practical construction of cells and safety envelopes is given in Section~\ref{sec:application_to_racing}.

\begin{algorithm}[t!]
\caption{Anytime Plug-and-Play MPC executed by agent $\firstagent$}
\label{alg:safety_envelope}
\begin{algorithmic}
\small
\STATE Given a feasible initialization at $\distime=0$, including a feasible planned trajectory $\tau_{\firstagent,0}^{0:N}$
\STATE Choose $\mathcal{E}_{\firstagent}^{N}$ such that Definition~\ref{def:safety_envelope} holds
\STATE Apply $u_{\firstagent,0|0}^{*}$
\FOR{$\distime = 1,2,\dots$}
    \FOR{each $\secondagent \in \mathcal{N}_{\firstagent}(\distime)$}
        \STATE Send $\tau_{\firstagent,\distime-1}^{0:N}$ to agent $\secondagent$
        \STATE Receive $\tau_{\secondagent,\distime-1}^{0:N}$ from agent $\secondagent$
    \ENDFOR
    \STATE Compute $\mathcal{P}^{N+1}_{\firstagent,\distime}$ such that Definition~\ref{def:partitions} holds
    \STATE Solve~\eqref{distributed_mpc_final} with $\mathcal{P}^{N+1}_{\firstagent,\distime}$ and $\mathcal{E}_{\firstagent}^{N}$
    \STATE Store the resulting planned position trajectory $\tau_{\firstagent,\distime}^{0:N}$
    \STATE Apply $u_{\firstagent,0|\distime}^{*}$
\ENDFOR
\end{algorithmic}
\end{algorithm}
A discussion, together with simulation and hardware results for Algorithm~\ref{alg:safety_envelope}, is provided in Section~\ref{sec:sim_and_hardware}.
\section{Application to autonomous driving}\label{sec:application_to_racing}
We demonstrate the effectiveness of the Anytime Plug-and-Play algorithm in a multi-agent autonomous driving scenario with a finite fleet of miniature race cars (agents) indexed by $\firstagent \in \mathcal{M} = \{1, \dots, M\}$. Each agent is modeled using the dynamic bicycle model~\cite{bodmer2024}, with state and input vectors
\begin{align*}
    x_{\firstagent} &= \begin{bmatrix}
        p_{\firstagent, x} & p_{\firstagent, y} & \gamma_{\firstagent} & v_{\firstagent, x} & v_{\firstagent, y} & \omega_{\firstagent}
    \end{bmatrix} \in \mathbb{R}^{n_x}, \\ 
    u_{\firstagent} &= \begin{bmatrix}
        \delta_{\firstagent} & T_{\firstagent}
    \end{bmatrix} \in \mathbb{R}^{n_u},
\end{align*}
where $(p_{\firstagent, x}, p_{\firstagent, y})$ is the position, $\gamma_{\firstagent}$ the heading, $(v_{\firstagent, x},v_{\firstagent, y})$ the body-frame velocities, and $\omega_{\firstagent}$ the angular velocity.  The control inputs are the steering angle $\delta_{\firstagent}$ and torque $T_{\firstagent}$. The agents race on a figure-eight track, where the intersection provides a realistic challenge by requiring collision avoidance and creating overtaking opportunities. 

\subsection{Anytime Plug-and-Play Multi-Trajectory model predictive contouring control (mt-MPCC)}
\label{sec:mt_mpcc}
The proposed Anytime PnP MPC controller~\eqref{distributed_mpc_final} provides collision avoidance guarantees and constraint satisfaction; however, it does not yet define a specific metric for highly dynamical autonomous driving performance. To address this, we augment the Anytime PnP MPC as a multi-trajectory MPC~\cite{saccani2022multitrajectory}, where we optimize over two coupled trajectories per agent $\firstagent$: 1) a safe trajectory, denoted with the superscript ($s$) that implements the Anytime PnP MPC formulation~\eqref{distributed_mpc_final}, and 2) an exploitation trajectory, denoted with the superscript ($e$) that implements an MPCC controller~\cite{faulwasser2009model} focusing on high-speed reference tracking without considering other agents. Both trajectories share the same first control input. Hence, the applied control action coincides with the one of the safe trajectory, and the recursive-feasibility, collision-avoidance, and constraint-satisfaction properties established for~\eqref{distributed_mpc_final} are retained.
We modify the classic MPCC cost function by augmenting it with a yaw-alignment term $\epsilon^{\gamma} = \gamma - \gamma_{\text{ref}}$, 
which promotes consistency with the track orientation. This modification allows for reduced contouring penalties, enabling overtaking maneuvers while avoiding unstable behavior.

The resulting Anytime PnP multi-trajectory MPCC optimization problem for agent $\firstagent$ is
\begin{subequations}\label{mtmpcc}
\begin{alignat}{2}
   &\min_{\substack{x_{\firstagent,0|\distime}^{s,e},\hdots,x_{\firstagent,N|\distime}^{s,e},\\
   u_{\firstagent,0|\distime}^{s,e},\hdots,u_{\firstagent,N-1|\distime}^{s,e},\\
   \bar{x}_{\firstagent|\distime}^{s},\,\bar{u}_{\firstagent|\distime}^{s}}}
   && \sum_{\step=0}^{N-1}
   J_{\firstagent}\!\left(x_{\firstagent,\step|\distime}^{s,e},
   u_{\firstagent,\step|\distime}^{s,e},
   \bar{x}_{\firstagent|\distime}^{s}\right)
   \label{mtmpcc:cost} \\
   &\text{subject to}
   && x_{\firstagent,0|\distime}^{s}
   =
   x_{\firstagent,0|\distime}^{e}
   =
   x_{\firstagent}(\distime)
   \label{mtmpcc:init_constr},\\
   &&
   & u_{\firstagent,0|\distime}^{s}
   =
   u_{\firstagent,0|\distime}^{e}
   \label{mtmpcc:init_input_constraint}, \\
   &&
   & x_{\firstagent,\stepnext|\distime}^{s,e}
   =
   f_\firstagent(x_{\firstagent,\step|\distime}^{s,e},u_{\firstagent,\step|\distime}^{s,e})
   \label{mtmpcc:dynamics_constr}, \\
      &&
   & x_{\firstagent,\step|\distime}^{s,e}\in\mathcal{X}_{\firstagent}
   \label{mtmpcc:state_constr}, \\
   &&
   & u_{\firstagent,\step|\distime}^{s,e}\in\mathcal{U}_{\firstagent}
   \label{mtmpcc:input_constr}, \\
   &&
   & p_{\firstagent,\step|\distime}^{s}
   =
   C_{\firstagent}x_{\firstagent,\step|\distime}^{s}
   \label{mtmpcc:position_constr}, \\
   &&
   & p_{\firstagent,\step|\distime}^{s}\in \mathcal{P}_{\firstagent,\step|\distime}
   \label{mtmpcc:hyperplane}, \\
   &&
   & \Delta p_{\firstagent,\step|\distime}^{s}\in \mathcal{E}_{\firstagent,\step}
   \label{mtmpcc:safety_envelope}, \\
   &&
   & x_{\firstagent,N|\distime}^{s} = \bar{x}_{\firstagent|\distime}^{s}
   \label{mtmpcc:terminal}, \\
   &&
   & \bar{x}_{\firstagent|\distime}^{s}
   =
   f_{\firstagent}(\bar{x}_{\firstagent|\distime}^{s},\bar{u}_{\firstagent|\distime}^{s})
   \label{mtmpcc:terminal_eq}, \\
   &&
   & (\bar{x}_{\firstagent|\distime}^{s},\bar{u}_{\firstagent|\distime}^{s})
   \in
   \mathcal{X}_{\firstagent}\times\mathcal{U}_{\firstagent}
   \label{mtmpcc:terminal_adm}
\end{alignat}
\end{subequations}
The constraints~\eqref{mtmpcc:dynamics_constr}, \eqref{mtmpcc:input_constr}, and~\eqref{mtmpcc:safety_envelope} are imposed for all $\step\in\mathbb{N}_0^{N-1}$, while~\eqref{mtmpcc:state_constr}, \eqref{mtmpcc:position_constr}, and~\eqref{mtmpcc:hyperplane} are imposed for all $\step\in\mathbb{N}_0^{N}$. 
Moreover, $\Delta p_{\firstagent,\step|\distime}^{s}
:=
p_{\firstagent,\stepnext|\distime}^{s}
-
p_{\firstagent,\step|\distime}^{s}$, $\forall \step\in\mathbb{N}_0^{N-1}$, and, by~\eqref{mtmpcc:terminal}, one has $p_{\firstagent,N|\distime}^{s}=C_\firstagent \bar x_{\firstagent|\distime}^{s}$.
The input constraints~\eqref{mtmpcc:input_constr} are polytopic, and the state constraints~\eqref{mtmpcc:state_constr} capture both physical limits and track boundaries, which can be formulated polytopically. Collision-avoidance constraints~\eqref{mtmpcc:hyperplane}--\eqref{mtmpcc:safety_envelope}, as well as the terminal conditions~\eqref{mtmpcc:terminal}--\eqref{mtmpcc:terminal_adm}, are enforced exclusively on the safe trajectory.

Because the safe trajectory satisfies the same structural constraints as~\eqref{distributed_mpc_final}, it inherits the recursive-feasibility and safety properties of the proposed Anytime PnP MPC scheme. The exploitation trajectory is used only to improve performance, while the shared first-input constraint~\eqref{mtmpcc:init_input_constraint} ensures that the closed-loop control action remains safety certified~\cite{saccani2022multitrajectory}.

\subsection{Anytime Plug-and-Play Ingredient Design} \label{sec:anytime_plug_and_play_design}
In this section, we discuss how the cells and safety envelopes introduced in Definition~\ref{def:partitions} and Definition~\ref{def:safety_envelope} can be designed, such that Proposition~\ref{prop:partitions} and Proposition~\ref{prop:safety_envelopes} hold. 

\subsubsection{Cells} \label{sec:partition_ex}
The cells $\mathcal{P}_{\firstagent , \distime}^{N+1}$ as defined in Definition~\ref{def:partitions} can be implemented as separating hyperplanes based on Voronoi partitioning~\cite{Rakovic2021}. The separating hyperplane between agents $\firstagent$ and $\secondagent$ at prediction step $\step$ is defined as:
\begin{align}
    \alpha^{(\firstagent,\secondagent)}_{(\step|\distime)} \cdot \mathbf{p} \leq \beta^{(\firstagent,\secondagent)}_{(\step|\distime)} - \frac{\epsilon_{(\firstagent,\secondagent)}}{2},
\end{align}
where $\mathbf{p} = \begin{bmatrix}p_x & p_y\end{bmatrix}$ and the hyperplane parameters are given by
\begin{align}
    \alpha^{(\firstagent,\secondagent)}_{(\step|\distime)} = \left( \frac{\mathbf{p}_{\secondagent,(\step|\distime)} - \mathbf{p}_{\firstagent,(\step|\distime)}}{||\mathbf{p}_{\secondagent,(\step|\distime)} - \mathbf{p}_{\firstagent,(\step|\distime)}||}\right)^\top, \\
    \beta^{(\firstagent,\secondagent)}_{(\step|\distime)} = \alpha^{(\firstagent,\secondagent)}_{(\step|\distime)} \cdot \frac{\mathbf{p}_{\firstagent,(\step|\distime)} +\mathbf{p}_{\secondagent,(\step|\distime)}}{2}.
\end{align}
Here, $\mathbf{p}_{\firstagent, (\step |\distime)}$ and $\mathbf{p}_{\secondagent, (\step |\distime)}$ denote the positions of agents $\firstagent$ and $\secondagent$ at the $\step$-th prediction step from the previous MPC solution. The superscript ($\firstagent$,$\secondagent$) indicates that the hyperplane is computed by agent 
$\firstagent$ with respect to agent $\secondagent$. Unlike classical Voronoi partitioning, where the separating hyperplane is located at the midpoint between the two agents, a safety margin $\epsilon_{(\firstagent, \secondagent)}$ is incorporated to ensure a minimum required separation. This margin satisfies $\epsilon_{(\firstagent,\secondagent)} \leq \|\mathbf{p}_{\firstagent,(\step|\distime)} - \mathbf{p}_{\secondagent,(\step|\distime)}\|$, thus ensuring collision avoidance when considering the volume of the agents.

In practice, particularly in highly dynamical driving scenarios, the use of such hyperplanes can restrict overtaking maneuvers. When two agents travel in the same direction, the resulting hyperplane becomes approximately perpendicular to their direction of motion. 

To alleviate this limitation and promote more natural overtaking behavior, we modify the hyperplane orientation. As shown in Figure~\ref{fig:take_over_45_deg}, the hyperplanes are deliberately tilted to create more space along the overtaking lane. For the figure-eight track used in our experiments, the orientation of the hyperplane is further adapted based on local track curvature. This strategy encourages overtaking along the inner side of a curve, as depicted in Figure~\ref{fig:curve_take_over_45_deg}.
\begin{figure}[H]
    \centering
  \subfloat[\label{fig:take_over_45_deg}]{%
       \includegraphics[width=0.5\linewidth]{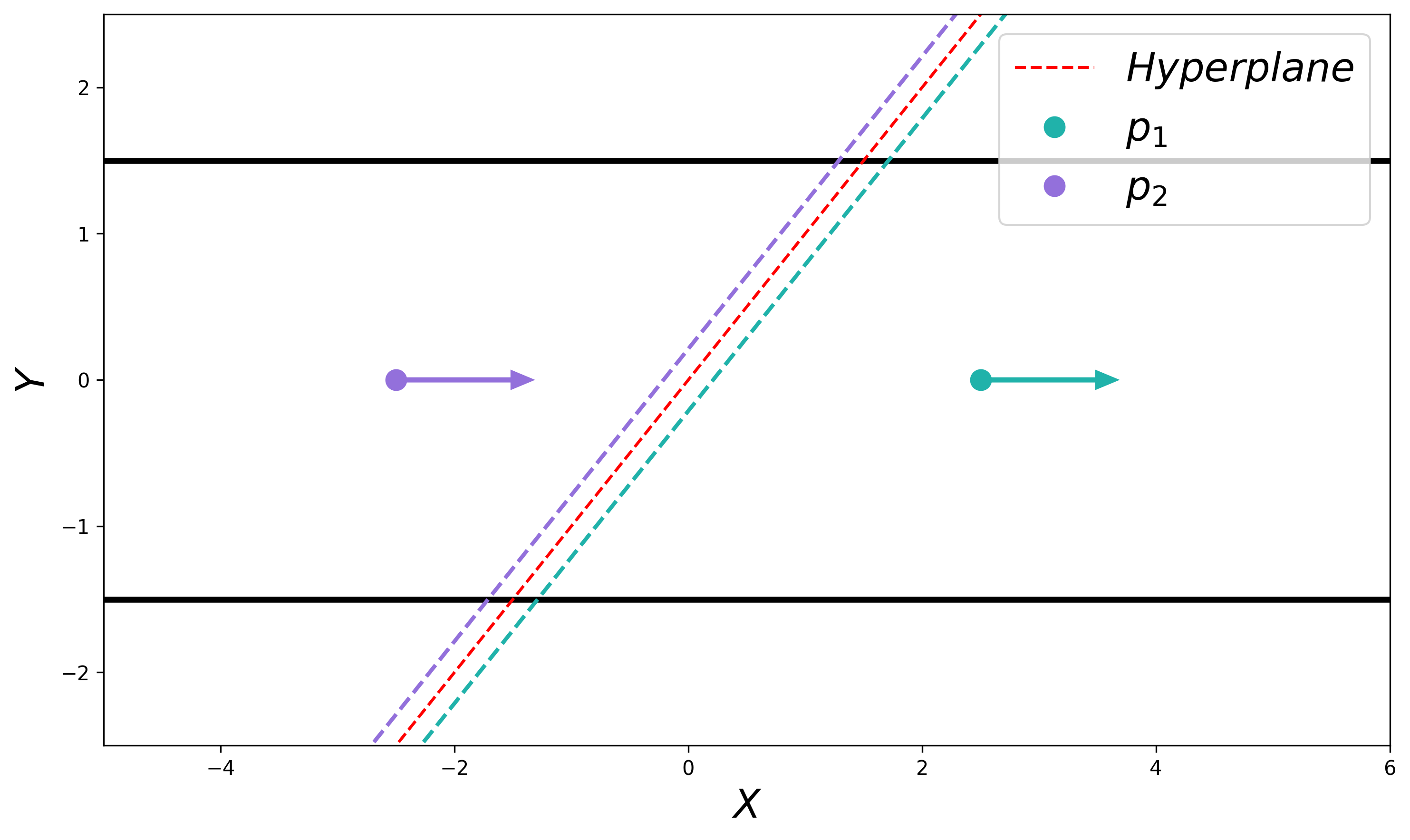}}
    \hfill
  \subfloat[\label{fig:curve_take_over_45_deg}]{%
        \includegraphics[width=0.5\linewidth]{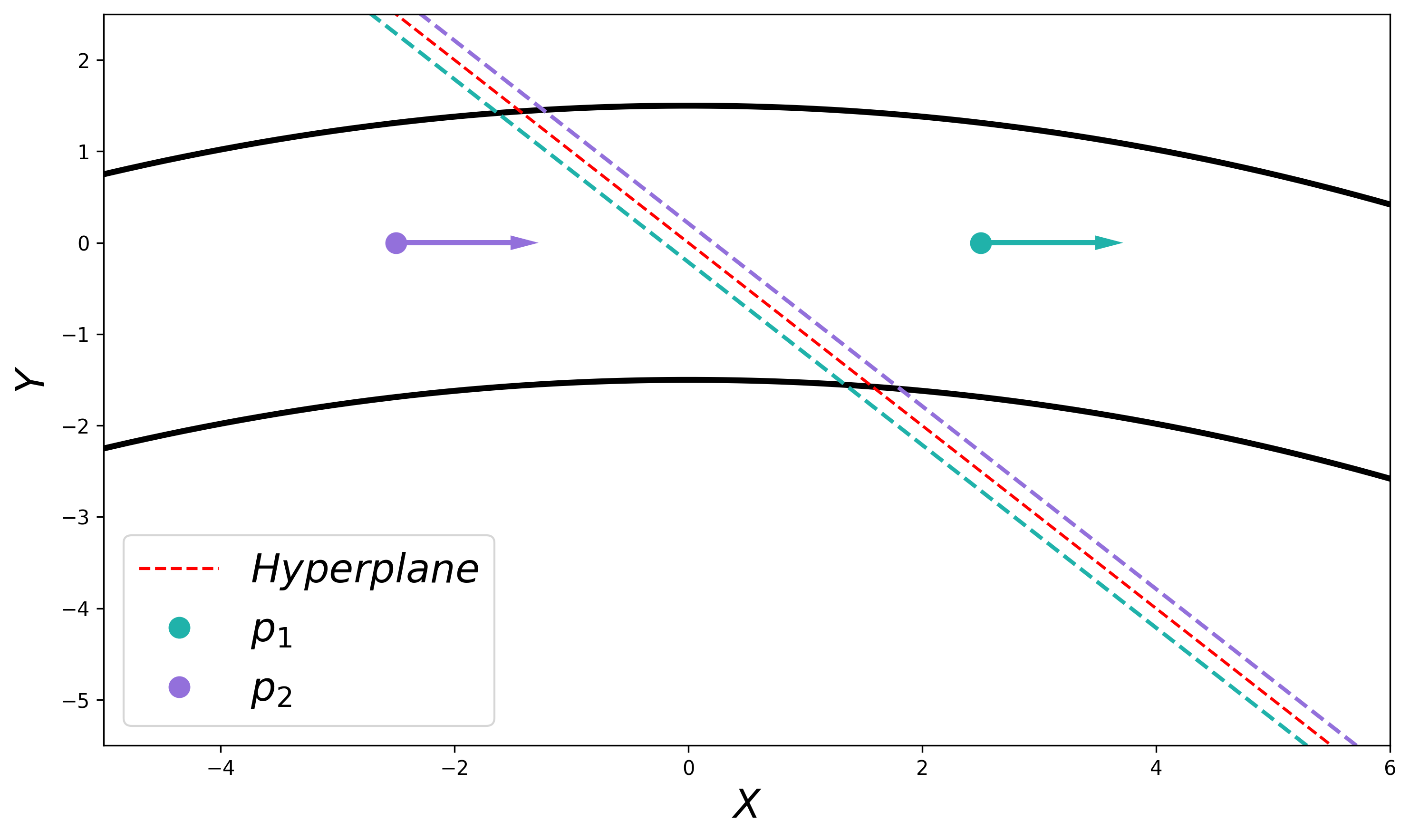}}
    \\
  \caption{\textbf{(a) Hyperplane angle:} Angling the hyperplane between neighboring agents results in less conservative behavior in regard to overtaking maneuvers. The hyperplane (red) at a 45$^\circ$ angle to the direction of movement (purple/blue arrows) between two agents reduces the difficulty of take-over maneuvers. The blue and purple hyperplanes account for a safety margin. \textbf{(b) Hyperplane angle based on track:} The angle of the hyperplane between neighboring agents depends on the curvature of the track. The hyperplane (red) at a 45$^\circ$ angle in a left turn. The blue and purple hyperplanes account for a safety margin.}
  \label{fig:hyperplane_angling} 
  
\vspace{-3mm}
\end{figure}

\subsubsection{Safety Envelope Design}\label{sec:safety_envelope_ex}
We design the safety envelopes as squares resulting in four hyperplane constraints per stage per agent. Higher-dimensional polytopes (hexagons) were evaluated as well, but showed no significant improvement in simulation. As stated in Definition~\ref{def:safety_envelope}, the $N$ safety envelopes are defined as a scaled awareness set, where the scalars $\alpha_{\firstagent, \step}$ are non-increasing along the prediction horizon, and sum up to one.
The scalars $\alpha_{\firstagent, \step}$ and the shape of the safety envelopes are considered tuning parameters. In the following, we present three possible implementations:
\begin{enumerate}
\item \textbf{Uniform safety envelopes:} Consider the square under-approximation of the awareness set. By scaling this under-approximation by the length of horizon $N$, uniform safety envelopes are created.

\item  \textbf{Continuously decreasing safety envelopes}: 
The safety envelope size monotonically decreases along the horizon.

\item \textbf{$\frac{1}{3}$ vs. $\frac{2}{3}$ safety envelopes:}
The first $\frac{1}{3}$ of the safety envelopes sum to one half of the awareness-set size, while the remaining $\frac{2}{3}$ sum to the other half, i.e.,
\begin{align*}
\sum_{\step=0}^{N-1} \alpha_{\firstagent,\step}
&=
\underbrace{\sum_{\step=0}^{\lfloor \frac{N-1}{3} \rfloor} \alpha_{\firstagent,\step}}_{= \frac{1}{2}}
+
\underbrace{\sum_{\step=\lfloor \frac{N-1}{3} \rfloor+1}^{N-1} \alpha_{\firstagent,\step}}_{= \frac{1}{2}}
= 1
\end{align*}
with
\begin{align*}
\alpha_{\firstagent,\step}=
\begin{cases}
\dfrac{1}{2\left(\lfloor \frac{N-1}{3} \rfloor+1\right)},
& \text{if } \step \le \lfloor \frac{N-1}{3} \rfloor,\\[2mm]
\dfrac{1}{2\left(N-\lfloor \frac{N-1}{3} \rfloor-1\right)},
& \text{otherwise}.
\end{cases}
\end{align*}
\end{enumerate}
The \textit{$\frac{1}{3}$ vs. $\frac{2}{3}$} safety envelopes resulted in improved speed compared to the uniform and continuously decreasing safety envelopes.

\section{Simulation and Hardware Experiments}\label{sec:sim_and_hardware}

We use the CRS software framework \cite{Carron_2023_crs, bodmer2024} to implement the approach proposed in Section~\ref{sec:application_to_racing}. We run the computations on an Intel Core i9-13900KS CPU (base frequency 3.20 GHz) and 48 GB of RAM.

We employ acados~\cite{Verschueren2021} using SQP-RTI as the MPC solver due to its capability to handle time-varying constraints efficiently. The MPC is run with a horizon length of 15 and at a sampling time of 20Hz. 
The CRS framework enables real-time communication between the local machine and the on-board microcontrollers and sensors via WiFi.
State estimation is performed using the framework’s built-in Extended Kalman Filter (EKF), and global pose information is provided by an external motion capture system\footnote[1]{\url{https://www.qualisys.com/cameras/arqus/}}.

\begin{figure*}[t!]
    \centering
    \includegraphics[trim={0cm 0.5cm 0cm 24cm}, clip, width=1.0\linewidth]{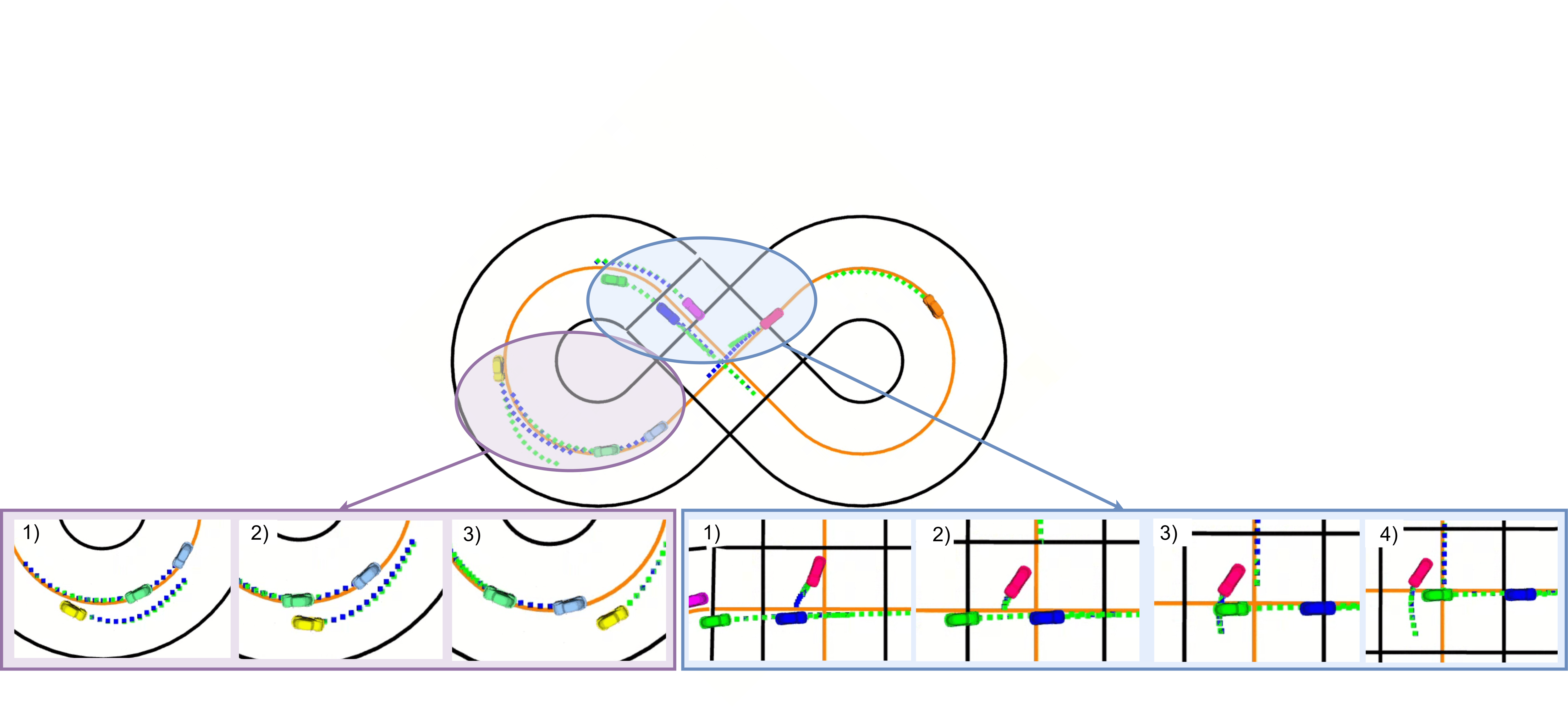}
    \caption{The colored dots denote the predicted positions along the MPC horizon associated with the safe (green) and exploitation (blue) trajectories. The purple bubble highlights an \textbf{oncoming traffic} scenario. The yellow agent moves off the centerline in order to avoid collisions with the oncoming green and blue agents. Initially, only the safe trajectory (green) diverts (see the yellow car in the purple circle of the upper image) from the centerline-reference, as the collision avoidance constraint persists, the exploitation trajectory adapts as well (1) - (3). The blue bubble highlights an \textbf{intersection} scenario. The red agent approaches a busy intersection. As the blue and green agents cross the intersection (1) - (3), the red agent slows down and diverts from the centerline reference in order to remain collision-free. Once the area is clear, the red agent speeds up again (4). }
    \label{fig:simulations}
    
\vspace{-3mm}
 \end{figure*}

\subsection{Simulation Experiments}
Simulations are conducted with up to eight agents operating on the figure-eight track. Agents are randomly initialized along the track, with a subset configured to race in the opposite direction to simulate head-on interactions.

Figures~\ref{fig:simulations} and~\ref{fig:sim_take_over} illustrate representative scenarios observed during simulation experiments involving eight agents. 
The reference path is given as the centerline (orange), while the predicted safe and exploitation trajectories are shown as green and blue dotted lines, respectively. Figure~\ref{fig:sim_take_over} further illustrates the safety envelopes (shown in red). Note that they are larger at the beginning and become smaller towards the end of the horizon, this is due to the choice of the \textit{$\frac{1}{3}$ vs. $\frac{2}{3}$} safety envelope, described in Section \ref{sec:safety_envelope_ex}.

In the following, we analyze three scenarios emerging in simulation. 
\subsubsection{Overtaking maneuver}
Consider two agents traveling in the same direction but at different speeds. We refer to the faster agent passing the slower agent as an \textit{overtaking maneuver}.
Figure~\ref{fig:sim_take_over} illustrates how the faster agent (purple) passes the slower agent (green) on the inner side, due to the curvature-dependent angling of the hyperplane.
\begin{figure}[t!]
    \centering
       \includegraphics[trim=0.1cm 0.1cm 0.2cm 0.2cm, clip, width=1.0\linewidth]{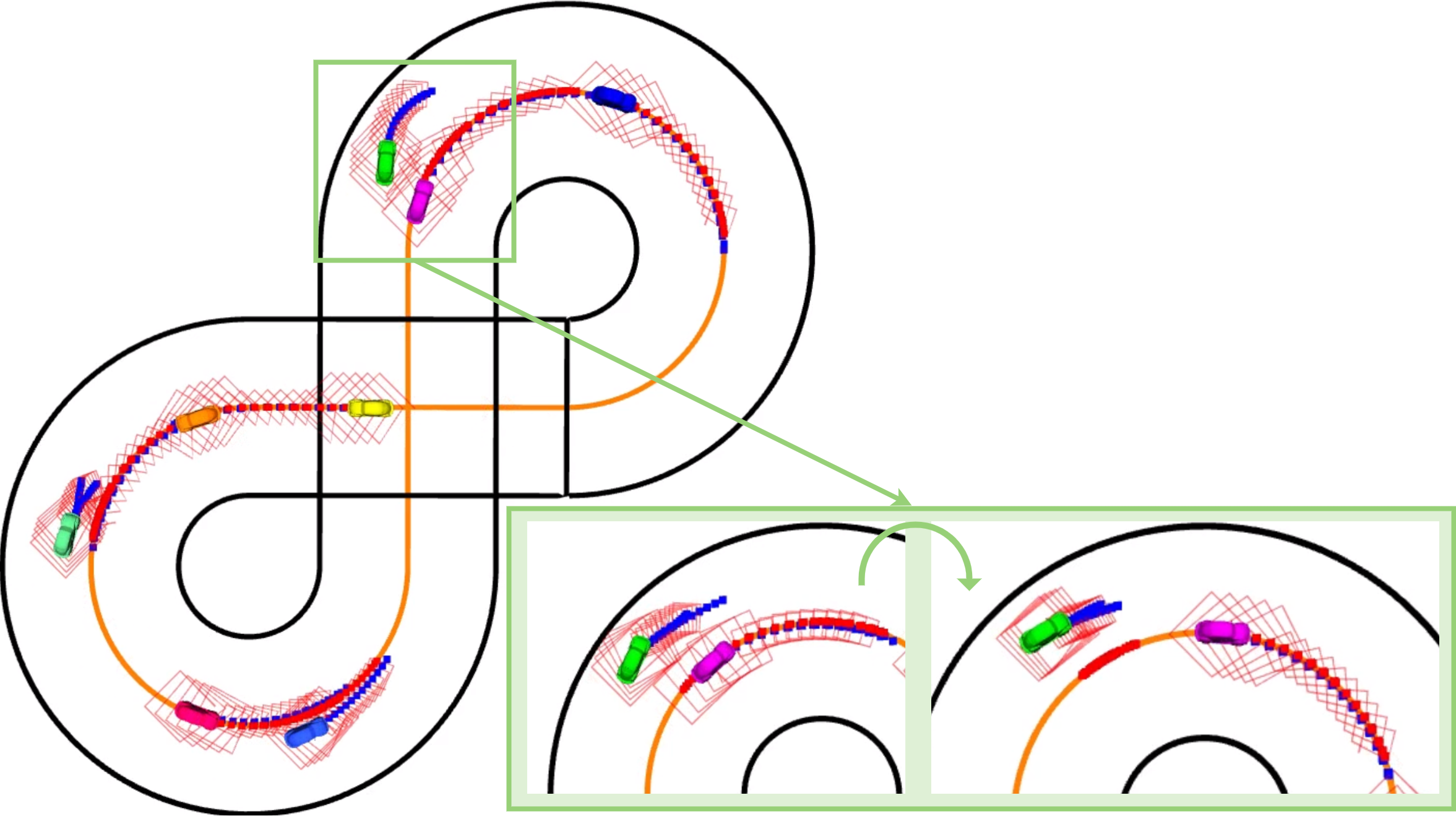}
   \caption{\textbf{Simulation experiment - Take over:} The purple agent overtakes the green agent on the inner rim of the curve, due to the 45$^\circ$ curve-dependent hyperplane angle. The green box highlights the progression (from left to right) of the simulation, focusing on the green and purple agents. The red boxes represent the safety envelopes that constrain the successive position increments along the prediction horizon.}
  \label{fig:sim_take_over} 
\vspace{-3mm}
\end{figure}
\subsubsection{Intersection}
The blue-highlighted section of Figure~\ref{fig:simulations} highlights one of the main difficulties with a figure-eight track - the intersection. The green agent is forced to \textit{swerve} in order to avoid colliding with the orange agent.

\subsubsection{Oncoming traffic}
With an increasing number of agents on the figure-eight track, avoiding oncoming traffic becomes increasingly difficult. The purple-highlighted section of Figure~\ref{fig:simulations} highlights that multiple cars need to avoid one another simultaneously.

\subsection{Hardware Experiments}
The proposed method is deployed on physical miniature race cars to validate its real-time performance and robustness. The experimental platform consists of custom-built 1/28th scale RC vehicles, as described in~\cite{Carron_2023_crs}.
Hardware experiments are conducted with between two and six agents operating on a figure-eight track. To evaluate the method under challenging interaction scenarios, some experiments feature agents driving in opposite directions, thereby inducing frequent oncoming traffic situations. 
Constraints are softened to ensure feasibility under real-world disturbances and uncertainties.

Figure~\ref{fig:6_agents} illustrates a challenging scenario in which multiple agents simultaneously approach the intersection of the figure-eight track. One agent, highlighted with a red circle, is emphasized to showcase its behavior during this high-density interaction. The driving direction of each agent is indicated by white arrows. Notably, the highlighted agent successfully navigates through the intersection, demonstrating the effectiveness of the proposed method in resolving complex multi-agent interactions without collisions.
A quantitative evaluation of the six-agent experiment is presented in Figure~\ref{fig:hardware_plots}. Figure~\ref{fig:min_dist} shows the minimum pairwise distance between any two agents throughout the experiment. The dashed line indicates the prescribed safety distance, below which a constraint violation occurs. We observe from the plot that collisions are avoided throughout the experiment, thereby ensuring safe operation.  Figure~\ref{fig:nr_neighbors} shows the number of neighboring agents a given agent has throughout the experiment. We note that the number of neighbors often varies over time, motivating the need for the proposed scheme that remains safe under a time-varying number of neighbors.
\begin{figure*}
    \centering
    \includegraphics[width=1.0\linewidth]{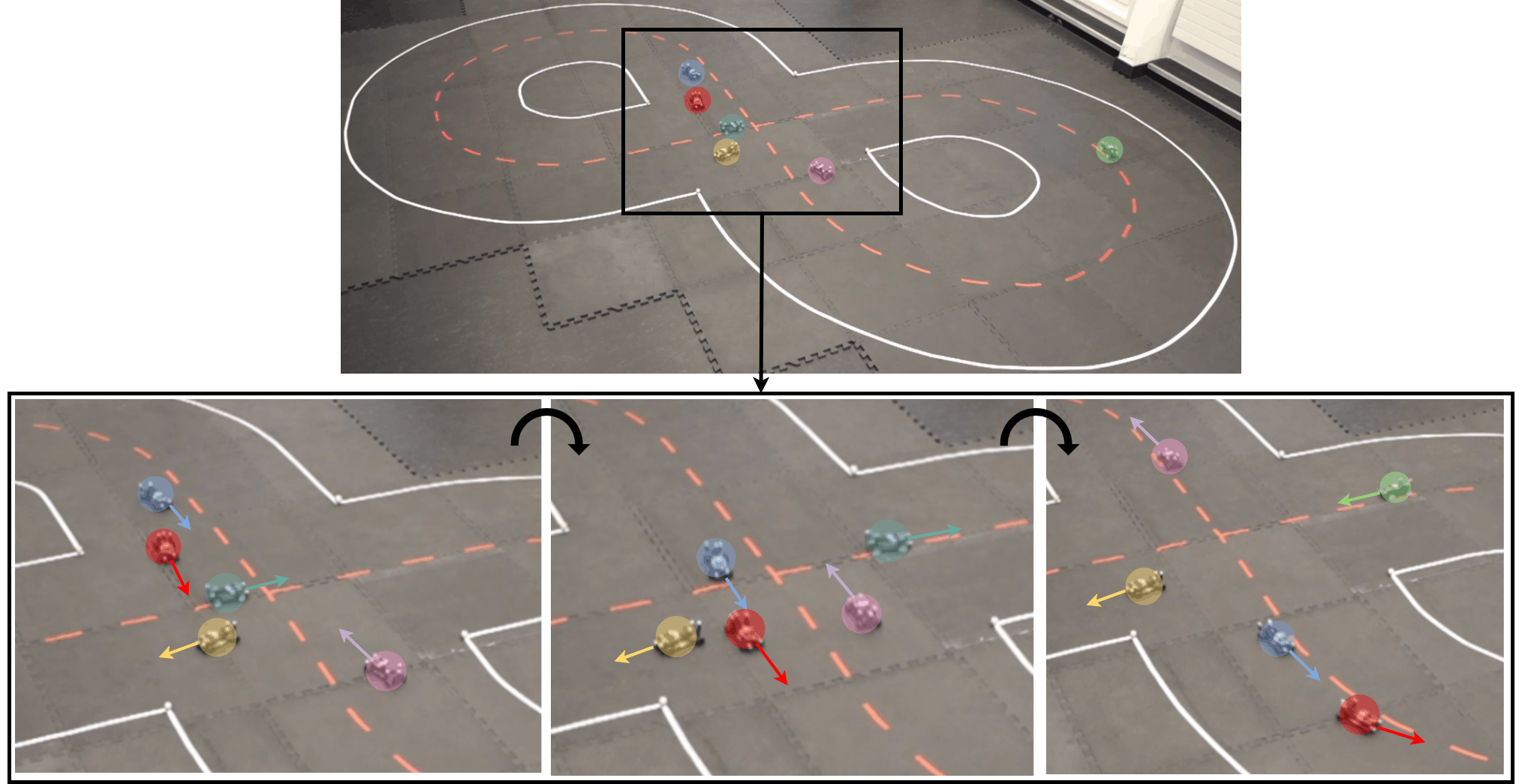}
    \caption{\textbf{Hardware experiment:} Each agent is highlighted with a different color, and its corresponding direction is illustrated by the same colored vector. Five agents approach the intersection at the same time and safely navigate through it.}
    \label{fig:6_agents}
\vspace{-3mm}
\end{figure*}
The hardware experiments demonstrate the real-world feasibility of the algorithm\footnote[2]{Hardware examples: \url{anytime-plug-and-play.github.io/}.}, allowing for fast, dynamical driving in multi-agent setups with collision avoidance. 

\begin{figure}[t]
    \centering 
  \subfloat[\label{fig:min_dist}]{%
       \includegraphics[width=1.0\linewidth]{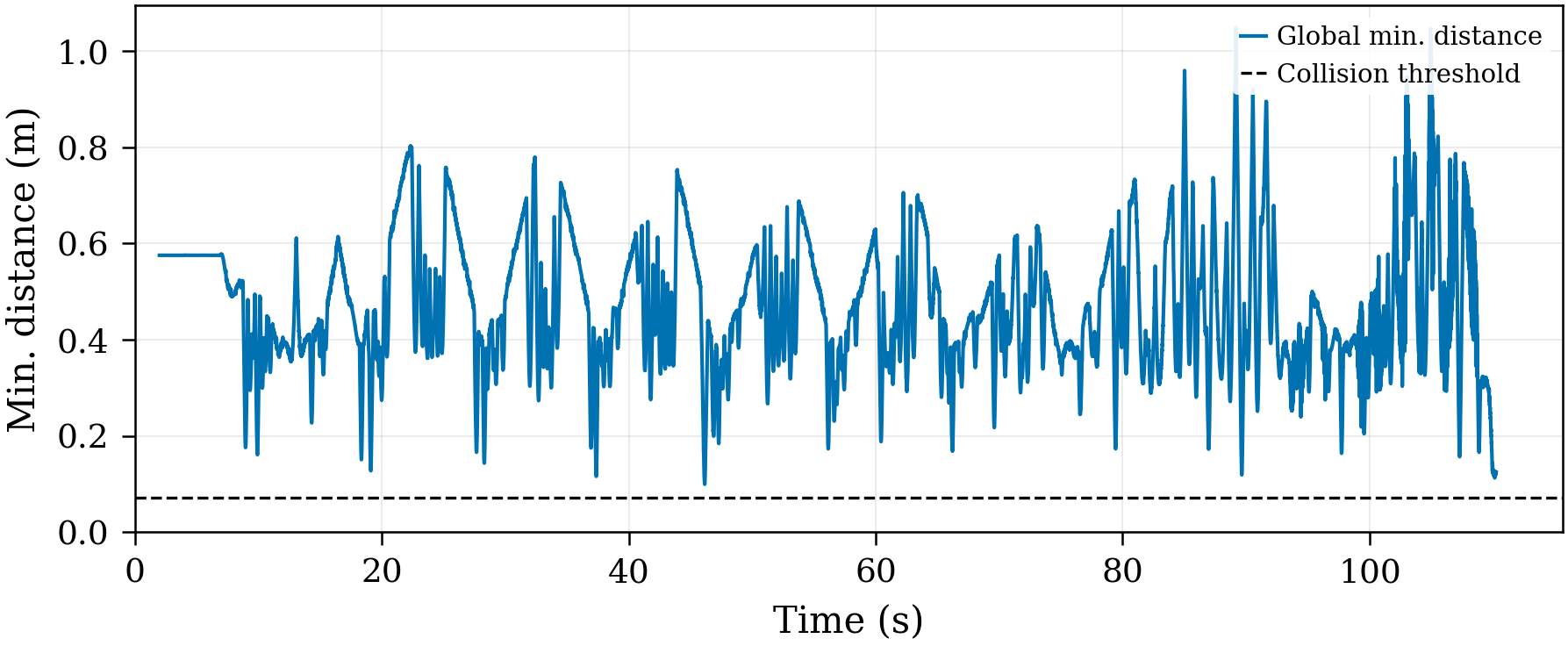}}
    \\
  \subfloat[\label{fig:nr_neighbors}]{%
        \includegraphics[width=1.0\linewidth]{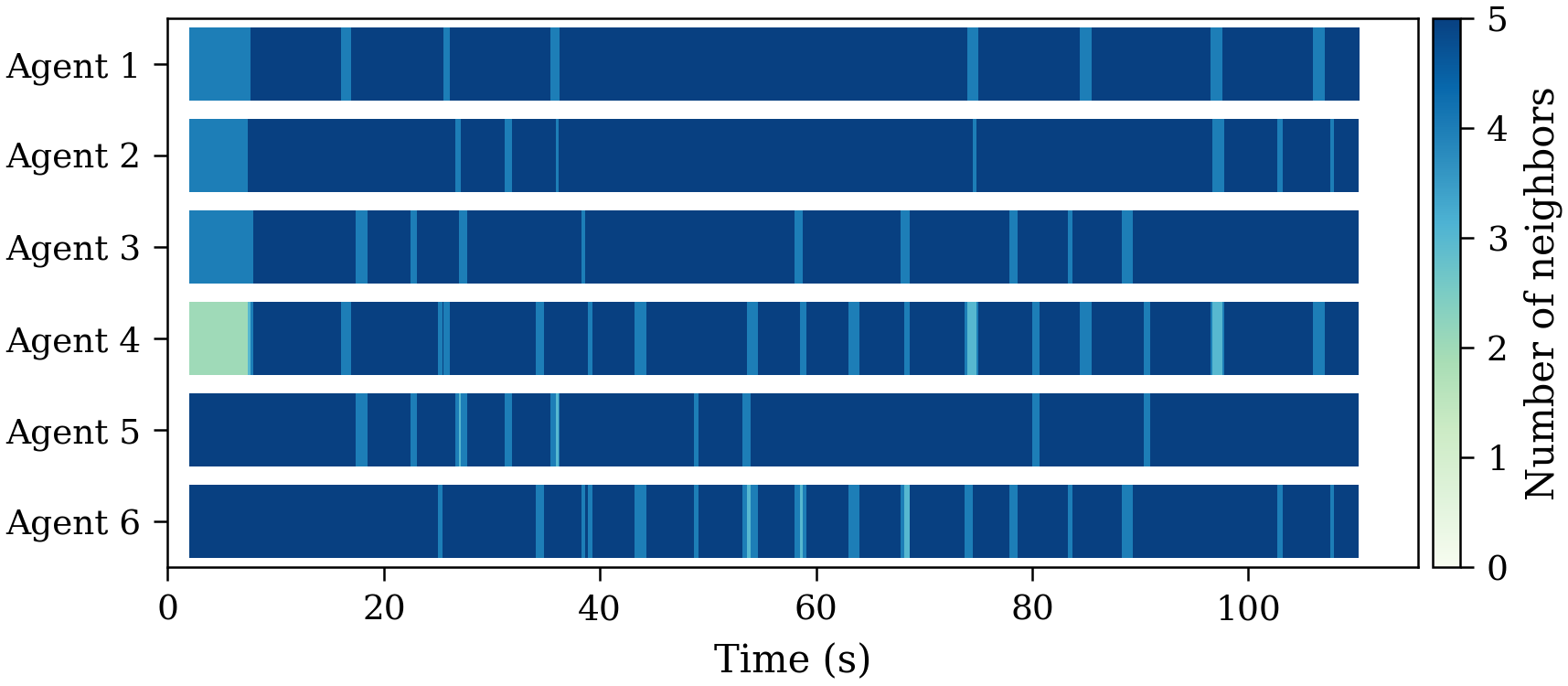}}
    \\
  \caption{\textbf{(a)} Minimal distance between any two agents throughout the duration of the experiment. The dashed line indicates the minimal distance (determined by the size of the agents) required, to avoid collisions. \textbf{(b)} Number of neighbors a given agent has throughout the experiment. The experiment was conducted using 6 agents, making the maximum number of neighbors 5.}
  \label{fig:hardware_plots} 
  
\vspace{-3mm}
\end{figure}

\section{Discussion and Limitations}
\label{sec: Discussion}
Prior distributed approaches often rely on centralized coordination, static topologies, fixed priorities, or request-based plug-and-play, limiting their applicability in robotic settings. In the following, we organize these methods into three categories: (i) approaches that introduce a global coordination among agents~\cite{Alrifaee_priority, ROQUE20203156}, (ii) approaches that introduce additional constraints to guarantee safety~\cite{FARINA20121088, Lucia2015contract, wiltz2022consistency}, 
and (iii) traffic-specific methods~\cite{BidirectionalMixedTraffic, AssistingDrivers, priority_based_highway_merging} (e.g., merging or overtaking). Game-theoretic approaches have also been proposed for multi-agent coordination, but since they address strategic interactions rather than safety-certified contract-based distributed MPC, we do not compare them directly with the proposed framework.  We compare the considered approaches in terms of their assumptions, communication requirements, and safety/feasibility guarantees with the proposed anytime plug-and-play formulation. Table~\ref{tab:comparison} summarizes the key differences. 

\begin{table*}[ht]
\caption{Comparison of approaches and their assumptions}
\label{tab:comparison}
\centering
\setlength{\tabcolsep}{5pt}
\renewcommand{\arraystretch}{1.15}
\resizebox{\textwidth}{!}{%
\begin{tabular}{@{} c l cccc ccc @{}}
\toprule
& & & & & & \multicolumn{3}{c}{As introduced by~\cite{SCATTOLINIreview}} \\
\cmidrule(l){7-9}
& \textbf{Approach}
& \textbf{Distributed}
& \textbf{Nonlinear Dynamics}
& \textbf{Plug-and-Play}
& \textbf{Safety Guarantees}
& \textbf{Non-cooperative}
& \textbf{Non-Iterative}
& \textbf{Non-Sequential} \\
\midrule

\multirow{3}{*}{\rotatebox[origin=c]{90}{\parbox{1.0cm}{\centering\small Global\\Coord.}}}
& Priority-based~\cite{Alrifaee_priority}
    & \checkmark & \checkmark & $\times$ & $\times$ & \checkmark & $\times$ & $\times$ \\
& Leader-Follower~\cite{ROQUE20203156}
    & \checkmark & \checkmark & $\times$ & $\times$ & $\times$ & \checkmark & $\times$ \\
& Distributed SQP~\cite{Zanon_semi_distributed_SQP_intersections}
    & partially  & \checkmark & $\times$ & \checkmark & \checkmark & \checkmark & \checkmark \\

\midrule

\multirow{6}{*}{\rotatebox[origin=c]{90}{\parbox{1.4cm}{\centering\small Constraint-based}}}
& Safety Tube~\cite{FARINA20121088}
    & \checkmark & $\times$ & $\times$ & $\times$ & \checkmark & \checkmark & \checkmark \\
& Consistency Constraint~\cite{wiltz2022consistency}
    & \checkmark & \checkmark & $\times$ & \checkmark & \checkmark & \checkmark & \checkmark \\
& Contract Constraint~\cite{Lucia2015contract}
    & \checkmark & \checkmark & request-based & \checkmark & \checkmark & \checkmark & \checkmark \\
& ECBF~\cite{goarin2024decentralizednonlinearmodelpredictive}
    & \checkmark & \checkmark & $\times$ & \checkmark & \checkmark & \checkmark & \checkmark \\
& CBF-induced QP~\cite{Dimos_CBF}
    & \checkmark & \checkmark & $\times$ & \checkmark & \checkmark & \checkmark & \checkmark \\
    & \cellcolor{blue!10}\textbf{Anytime Plug-and-Play (ours)}
    & \cellcolor{blue!10}\pmb{\checkmark} & \cellcolor{blue!10}\pmb{\checkmark} & \cellcolor{blue!10}\textbf{anytime} & \cellcolor{blue!10}\pmb{\checkmark} & \cellcolor{blue!10}\pmb{\checkmark} & \cellcolor{blue!10}\pmb{\checkmark} & \cellcolor{blue!10}\pmb{\checkmark} \\

\midrule



\multirow{3}{*}{\rotatebox[origin=c]{90}{\parbox{1.0cm}{\centering\small Traffic\\App.}}}
& Assisting Take-Over Maneuver~\cite{AssistingDrivers}
    & \checkmark & $\times$ & request-based & $\times$ & $\times$ & \checkmark & $\times$ \\
& Bidirectional Mixed-Traffic Take-Over~\cite{BidirectionalMixedTraffic}
    & \checkmark & $\times$ & anytime & $\times$ & $\times$ & \checkmark & $\times$ \\
& Highway Merging~\cite{priority_based_highway_merging}
    & $\times$ & \checkmark & n/a & $\times$ & $\times$ & \checkmark & \checkmark \\

\bottomrule
\end{tabular}}
\end{table*}

\subsection{Global Coordination}
A common strategy for collision avoidance in distributed multi-agent systems is to impose a global coordination structure. In~\cite{Alrifaee_priority}, agents are assigned priorities and solve their optimization problems sequentially while considering higher-priority solutions, but this requires global knowledge. Similarly,\cite{ROQUE20203156} employs leader–follower formation control, where the leader acts as a central coordinator. While this ensures collision-free formations, it relies on a fixed hierarchy and communication network. The approach in\cite{Zanon_semi_distributed_SQP_intersections} proposes a Sequential Quadratic Programming (SQP)-based formulation, where a central node manages intersections by collecting locally optimized time slots from vehicles and distributing them as initialization for local MPC problems. Although these methods provide strong guarantees, they depend on centralized coordination or tight communication loops, which limits their flexibility in plug-and-play scenarios.
Relatedly,~\cite{rudolf_reiter_hierarchical} propose a hierarchical RL-NMPC planner for strategic autonomous racing, but it assumes fixed opponents and restrictive prediction/perfect tracking, effectively relying on global scene awareness rather than plug-and-play, graph-based local coordination.

\subsection{Additional Constraints}
Another strategy for guaranteeing collision avoidance is to add constraints to local optimization problems. The proposed work falls in this category. In~\cite{FARINA20121088, Lucia2015contract}, contracts define sets that are guaranteed to contain trajectories of coupling variables, ensuring recursive feasibility through additional constraints. This enables plug-and-play operation~\cite{Lucia2015contract}, but feasibility checks may still block plug-in or plug-out requests. Similarly,~\cite{wiltz2022consistency} constructs consistency sets around reference trajectories and iteratively adjusts them until terminal and constraint requirements are satisfied. However, since these sets are not updated at every optimization step, they cannot handle dynamic network changes. Overall, while such methods provide distributed collision avoidance through contract or consistency mechanisms, they do not support anytime plug-and-play operation.
A related line of work employs control barrier functions (CBFs). For example,~\cite{goarin2024decentralizednonlinearmodelpredictive} uses exponential CBFs based on relative state estimates, while~\cite{Dimos_CBF} formulates CBF-induced QPs with auxiliary shared variables. Both approaches provide strong safety guarantees but assume full-state information or connected communication graphs, making them unsuitable for scenarios with variable agent participation.

\subsection{Traffic Scenarios}
Obstacle avoidance has been widely studied in traffic settings, where vehicles are treated as dynamic obstacles. Many works focus on merging and overtaking maneuvers~\cite{BidirectionalMixedTraffic, AssistingDrivers, priority_based_highway_merging}.  
In~\cite{BidirectionalMixedTraffic}, vehicle-to-vehicle communication is used to estimate oncoming traffic, combining autonomous and human-driven vehicles modeled via car dynamics. Similarly,~\cite{AssistingDrivers} considers multi-vehicle overtaking with shared information about traffic and intended maneuvers. While effective for overtaking, these approaches rely on communication among all involved vehicles and do not generalize to broader multi-agent scenarios or provide safety guarantees. Highway merging is addressed in~\cite{priority_based_highway_merging}, where a central unit assigns priorities to determine merge sequences. Although successful for this setting, the reliance on priority assignments and centralized coordination limits general applicability. Overall, traffic-specific approaches address well-defined scenarios but often assume cooperation, central coordination, or do not focus on formal safety guarantees, restricting their use in general distributed multi-agent systems.

\subsection{Discussion and Limitations}
While these alternatives offer valuable insights and benefits in specific contexts, it is important to stress that a direct comparison is not entirely fair. Each framework is tailored to a different trade-off between optimality, feasibility, decentralization, and flexibility, rendering direct comparisons difficult. The proposed method targets settings where real-time operation, distributed implementation, and agents joining or leaving the network dynamically are essential, and achieves a favorable balance within these objectives. Despite promising results, some limitations merit further investigation. Safety envelopes, while enabling decoupled collision avoidance, introduce conservatism by limiting per-step progress. Adaptive envelope tuning, potentially via learning-based methods, could better balance safety and performance. Moreover, although the framework relies only on neighbor-to-neighbor communication, it still requires information exchange. Moving toward communication-free distributed architectures remains an open and compelling research direction.

\section{Conclusion}
This work introduced a distributed control framework for multi-agent systems that guarantees collision-free operation with anytime plug-and-play capability. By decoupling collision avoidance constraints through cells and safety envelopes, the method enables fully distributed execution without centralized coordination.
We established recursive feasibility, proving safe anytime PnP operations, and validated the approach in simulation and hardware experiments. To our knowledge, this is the first distributed collision-avoidance framework that avoids request-based or pre-negotiated plug-and-play, laying the groundwork for scalable, flexible multi-agent applications with formal safety guarantees.

\appendix
\subsection{Proof of Proposition~\ref{prop:partitions}} \label{App:proof_partitions}
The claim follows directly from Definition~\ref{def:partitions}. Indeed, if
\[
p_{\firstagent,\step|\distime}\in \mathcal{P}_{\firstagent,\step|\distime},
\qquad
p_{\secondagent,\step|\distime}\in \mathcal{P}_{\secondagent,\step|\distime},
\]
and the inflated cells are disjoint as in~\eqref{eq:compatible_cells}, then the two positions must satisfy $\|p_{\firstagent,\step|\distime}-p_{\secondagent,\step|\distime}\| \ge \epsilon$ and using~\eqref{eq:collision_avoidance}, it follows that $h(x_{\firstagent,\step|\distime},x_{\secondagent,\step|\distime}) \ge 0$.
\subsection{Proof of Lemma~\ref{lem:cell_existence}} \label{App:proof_lem_partitions}
For each agent $\firstagent\in\mathcal M$ and each prediction step $\step\in\mathbb N_0^{N}$, define
\[
\mathcal P_{\firstagent,\step|\distime}:=\{p_{\firstagent,\step|\distime}\}.
\]
Each set is convex and, by assumption, is contained in the admissible position space
$\{C_\firstagent x_\firstagent \mid x_\firstagent\in\mathcal X_\firstagent\}$.

Now fix any $\firstagent\in\mathcal M$, any
$\secondagent\in\mathcal N_{\firstagent}(\distime)$, and any
$\step\in\mathbb N_0^{N}$. Since
\[
h(x_{\firstagent,\step|\distime},x_{\secondagent,\step|\distime}) \ge 0,
\]
by~\eqref{eq:collision_avoidance} we have
\[
\|p_{\firstagent,\step|\distime}-p_{\secondagent,\step|\distime}\| \ge \epsilon.
\]
Therefore,
\[
\big(\{p_{\firstagent,\step|\distime}\}\oplus\mathcal B_{\epsilon/2}\big)
\cap
\big(\{p_{\secondagent,\step|\distime}\}\oplus\mathcal B_{\epsilon/2}\big)
=
\emptyset,
\]
because $\mathcal B_{\epsilon/2}$ is an open ball. Hence~\eqref{eq:compatible_cells} holds, and the resulting family of cells is compatible.

\subsection{Proof of Proposition~\ref{prop:non_neighbor_safe_set}} \label{App:proof_cells1}
Since $\secondagent \notin \mathcal{N}_{\firstagent}(\distime)$, by~\eqref{eq:neighbor_set} it holds that
\[
\bar{\mathcal{R}}(p_{\firstagent,0|\distime})
\cap
\bar{\mathcal{R}}(p_{\secondagent,0|\distime})
=
\emptyset.
\]
If~\eqref{eq:pairwise_safe_set_constraint} holds, then, for every $\step\in\mathbb{N}_0^{N}$,
\[
p_{\firstagent,\step|\distime}+ \mathcal{B}_{\epsilon/2}
\subseteq
\bar{\mathcal{R}}(p_{\firstagent,0|\distime}),
\quad
p_{\secondagent,\step|\distime}+ \mathcal{B}_{\epsilon/2}
\subseteq
\bar{\mathcal{R}}(p_{\secondagent,0|\distime}).
\]
Hence these two balls are disjoint, and therefore
\[
\|p_{\firstagent,\step|\distime}-p_{\secondagent,\step|\distime}\|\ge \epsilon,
\qquad
\forall \step\in\mathbb{N}_0^{N}.
\]
Using~\eqref{eq:collision_avoidance}, the claim follows.

\subsection{Proof of Proposition~\ref{prop:safety_envelopes}} \label{App:proof_cells2}
We first prove~\eqref{eq:traj_in_safe}. For every $m\in\mathbb{N}_1^{N}$,
\[
p_{\firstagent,m|\distime}-p_{\firstagent,0|\distime}
=
\sum_{\step=0}^{m-1}
\big(p_{\firstagent,\step+1|\distime}-p_{\firstagent,\step|\distime}\big)
\in
\bigoplus_{\step=0}^{m-1}\mathcal{E}_{\firstagent,\step},
\]
where the inclusion follows from~\eqref{eq:constr_safety_envelope}. Hence,
\[
p_{\firstagent,m|\distime}-p_{\firstagent,0|\distime}
\in
\bigoplus_{\step=0}^{m-1}\mathcal{E}_{\firstagent,\step}
\subseteq
\bigoplus_{\step=0}^{N-1}\mathcal{E}_{\firstagent,\step}
\subseteq
\mathcal{A}(\mathbf{0}_3),
\]
by~\eqref{eq:sum_of_safety_envelopes}. Therefore,
\[
p_{\firstagent,m|\distime}
\in
\mathcal{A}(\mathbf{0}_3)+ p_{\firstagent,0|\distime}
=
\mathcal{A}(p_{\firstagent,0|\distime}).
\]
Since also $p_{\firstagent,0|\distime}\in \mathcal{A}(p_{\firstagent,0|\distime})$, \eqref{eq:traj_in_safe} follows.

To prove~\eqref{eq:shift_traj_in_safe}, fix any $\ell\in\mathbb{N}_0^{N-1}$ and any $m\in\mathbb{N}_{\ell+1}^{N}$. Then
\begin{align*}
    p_{\firstagent,m|\distime}-p_{\firstagent,\ell|\distime}
&=
\sum_{\step=\ell}^{m-1}
\big(p_{\firstagent,\step+1|\distime}-p_{\firstagent,\step|\distime}\big)
\\ & \in
\bigoplus_{\step=\ell}^{m-1}\mathcal{E}_{\firstagent,\step}
\subseteq
\bigoplus_{\step=0}^{N-1}\mathcal{E}_{\firstagent,\step}
\subseteq
\mathcal{A}(\mathbf{0}_3).
\end{align*}
Therefore, $p_{\firstagent,m|\distime}
\in
\mathcal{A}(\mathbf{0}_3)+ p_{\firstagent,\ell|\distime}
=
\mathcal{A}(p_{\firstagent,\ell|\distime})$.
Since also $p_{\firstagent,\ell|\distime}\in \mathcal{A}(p_{\firstagent,\ell|\distime})$, \eqref{eq:shift_traj_in_safe} follows.

\bibliographystyle{IEEEtran}
\bibliography{bibliography}
\addtolength{\textheight}{-3cm}   
%

\vspace{-1.2cm}
\begin{IEEEbiography}[{\includegraphics[width=1in,height=1.25in,clip,keepaspectratio]{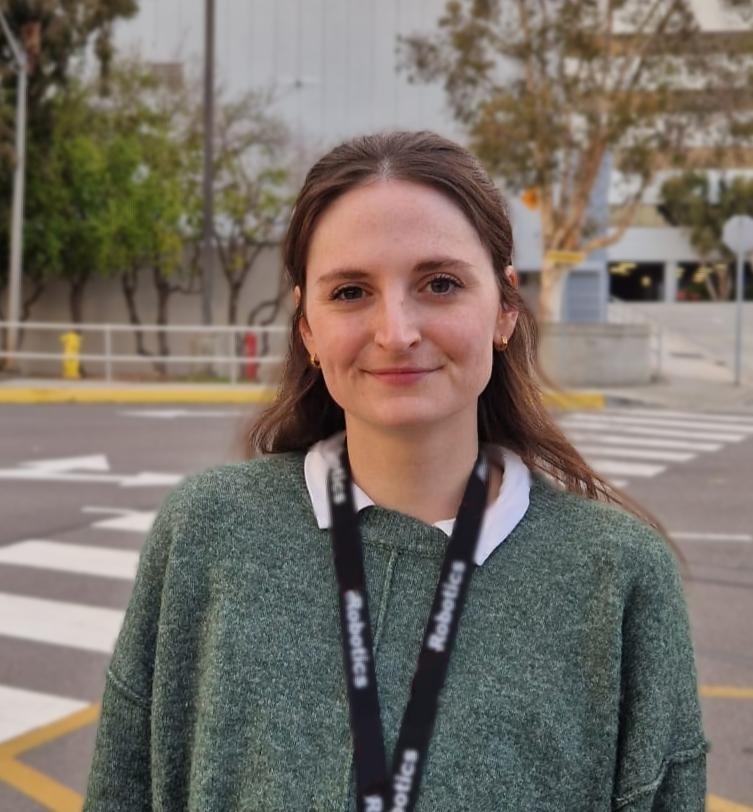}}] {Sabrina Bodmer} received her bachelor’s degree in electrical engineering in 2021 and her master’s degree in robotics, systems, and control in 2025, from ETH Zürich, Switzerland. She is currently working toward a Ph.D. degree in mechanical engineering at the ETH Zürich as well. Her research interests include the areas of multi-agent and distributed control, as well as learning-based control.
\end{IEEEbiography}

\vspace{-1.2cm}

\begin{IEEEbiography}[{\includegraphics[width=1in,height=1.25in,clip,keepaspectratio]{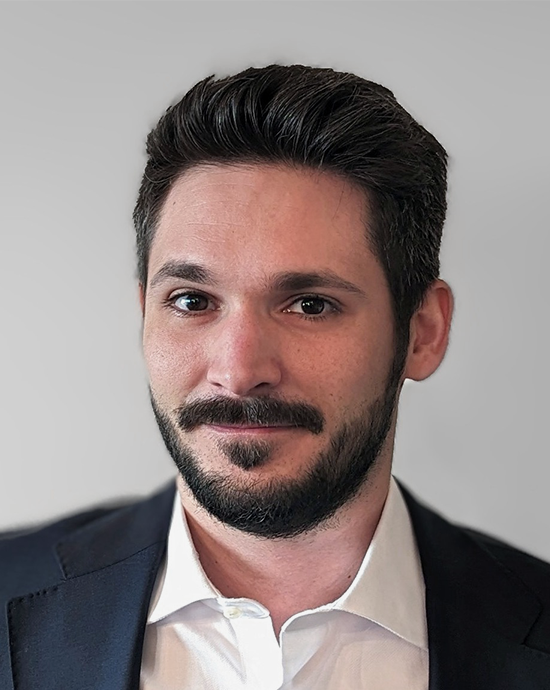}}]{Danilo Saccani} received the Ph.D. degree in Control Engineering from Politecnico di Milano, Italy, in 2023. Since 2023, he has been a Postdoctoral Researcher at EPFL, Switzerland. His research interests include modeling, optimization, and control of dynamical systems, with a focus on distributed control, model predictive control, neural network-based control, and autonomous vehicles.\end{IEEEbiography}

\vspace{-1.2cm}

\begin{IEEEbiography}[{\includegraphics[width=1in,height=1.25in,clip,keepaspectratio]{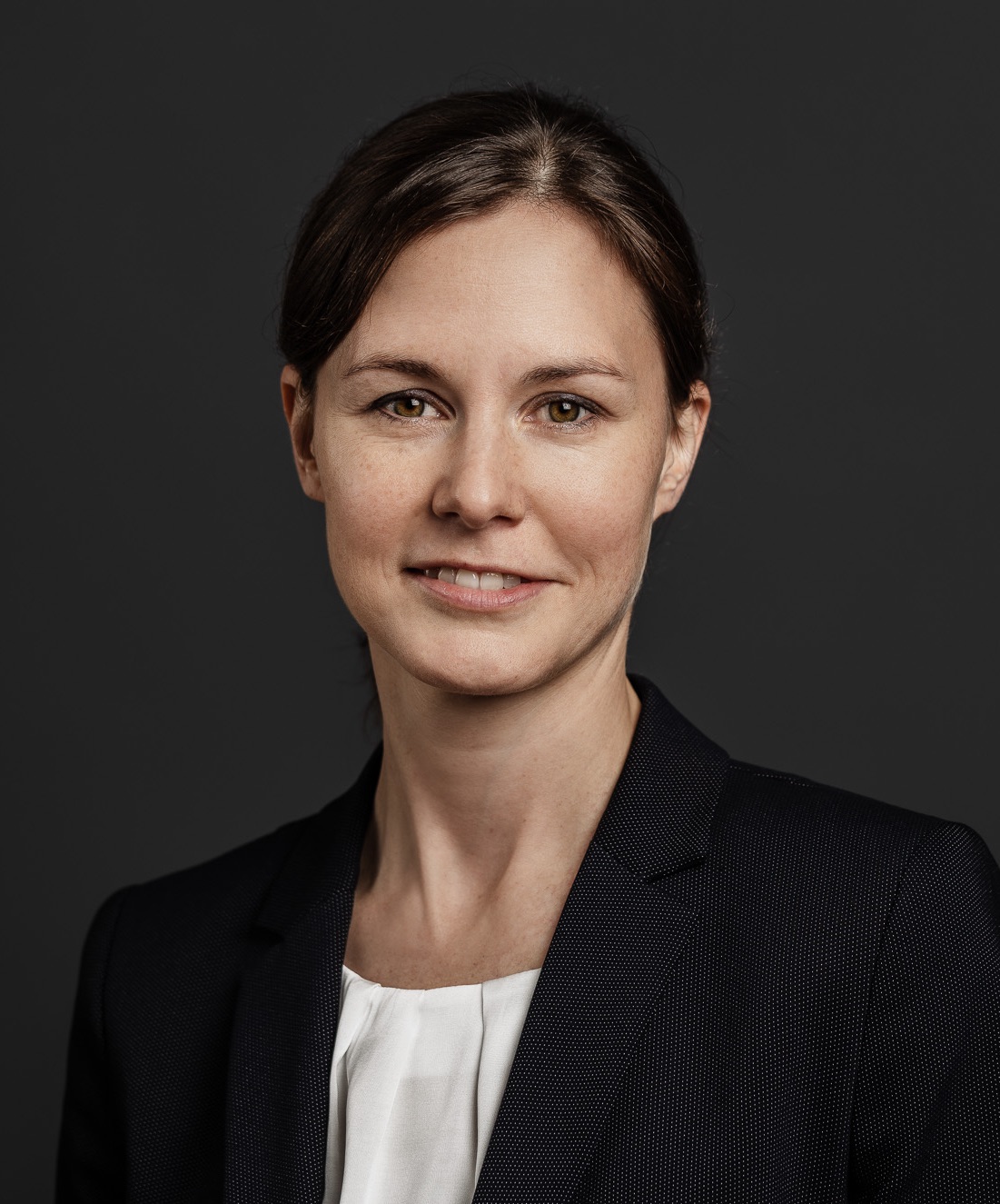}}]{Melanie N. Zeilinger} is an Associate Professor at ETH Zurich, Switzerland. She received the Diploma degree in engineering cybernetics from the University of Stuttgart, Germany, in 2006, and the Ph.D. degree with honors in electrical engineering from ETH Zurich, Switzerland, in 2011. From 2011 to 2012 she was a Postdoctoral Fellow with the Ecole Polytechnique Federale de Lausanne (EPFL), Switzerland. She was a Marie Curie Fellow and Postdoctoral Researcher with the Max Planck Institute for Intelligent Systems, Tubingen, Germany until 2015 and with the Department of Electrical Engineering and Computer Sciences at the University of California at Berkeley, CA, USA, from 2012 to 2014. From 2018 to 2019 she was a professor at the University of Freiburg, Germany. Her current research interests include safe learning-based control, as well as distributed control and optimization, with applications to robotics and human-in-the loop control.
 \end{IEEEbiography}

 \vspace{-1.2cm}

\begin{IEEEbiography}[{\includegraphics[width=1in,height=1.25in,clip,keepaspectratio]{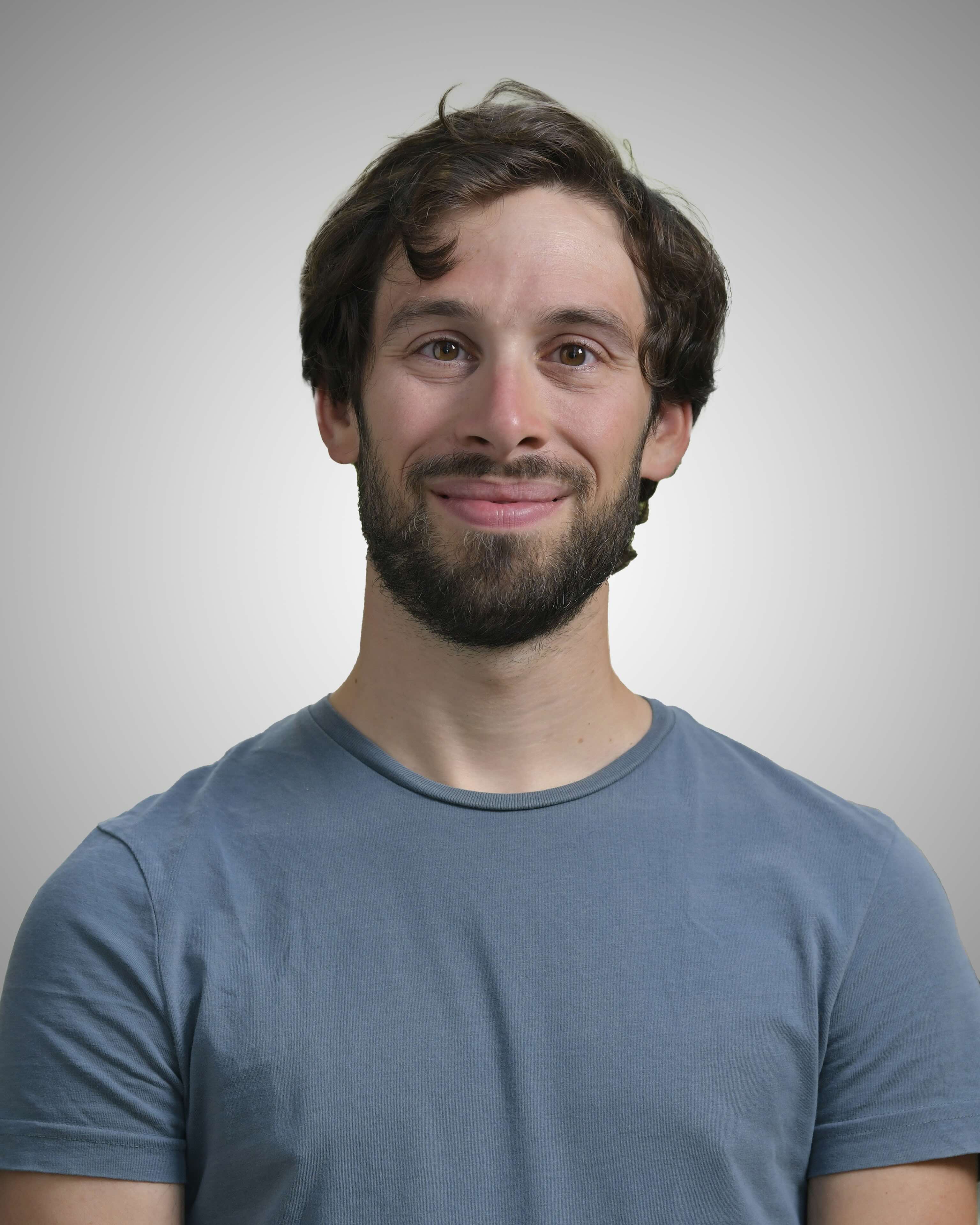}}] {Andrea Carron} received his Ph.D. in control engineering from the University of Padova in 2016. He is a Senior Scientist at ETH Zurich and previously held a postdoctoral position with the Intelligent Control Systems Group. He has also been a Visiting Researcher at UC Riverside, the Max Planck Institute, and UC Santa Barbara. His research interests include safe learning, learning-based control, multi-agent systems, and robotics.
\end{IEEEbiography}
\vfill

\end{document}